\documentclass[11pt, a4paper]{article}
\usepackage[margin=4cm]{geometry}
\usepackage{hyperref}
\usepackage{color}
\usepackage[round]{natbib}
\usepackage{array}
\usepackage{tikz}
\usetikzlibrary{shapes, calc, arrows}

\usepackage{graphicx}

\usepackage{amsmath, amsfonts, amsthm, amssymb, bbm, bm}
\usepackage{mathrsfs}
\usepackage{enumerate}
\usepackage{tabularx}
\usepackage{multicol}
\usepackage{multirow}

\newcommand{\E}{\mathbb{E}}
\DeclareMathOperator{\Var}{Var}

\newcommand{\G}{\mathcal{G}}
\newcommand{\X}{\mathfrak{X}}

\DeclareMathOperator{\pa}{pa}
\DeclareMathOperator{\ch}{ch}

\DeclareMathOperator{\an}{an}

\newcommand\indep{\protect\mathpalette{\protect\independenT}{\perp}}
\def\independenT#1#2{\mathrel{\rlap{$#1#2$}\mkern2mu{#1#2}}}

\theoremstyle{plain}
\newtheorem{lem}{Lemma}[section]
\newtheorem{thm}[lem]{Theorem}
\newtheorem{prop}[lem]{Proposition}
\newtheorem{cor}[lem]{Corollary}

\theoremstyle{definition}
\newtheorem{dfn}[lem]{Definition}
\newtheorem{rmk}[lem]{Remark}
\newtheorem{exm}[lem]{Example}

\newcommand{\benum}{\begin{enumerate}}
\newcommand{\eenum}{\end{enumerate}}

\newcommand{\bitem}{\begin{itemize}}
\newcommand{\eitem}{\end{itemize}}

\newcommand{\barr}{\begin{array}}
\newcommand{\earr}{\end{array}}

\newcommand{\bmat}{\begin{pmatrix}}
\newcommand{\emat}{\end{pmatrix}}

\newcommand{\blist}{\renewcommand{\labelenumi}{\textbf{\arabic{enumi}}.} \begin{enumerate}}
\newcommand{\elist}{\end{enumerate} \renewcommand{\labelenumi}{\arabic{enumi}.}}

\newcommand{\bs}{\boldsymbol}

\def\bal#1\eal{\begin{align*}#1\end{align*}}

\usepackage{euscript}

\DeclareFontFamily{OT1}{pzc}{}
\DeclareFontShape{OT1}{pzc}{m}{it}%
              {<-> s * [0.900] pzcmi7t}{}
\DeclareMathAlphabet{\mathpzc}{OT1}{pzc}%
                                 {m}{it}

\newcommand{\M}{{\mathcal{M}}}
\newcommand{\D}{{\mathcal{D}}}
\renewcommand{\X}{{\mathcal{X}}}

\renewcommand{\H}{\mathcal{H}}

\title{Graphs for margins of Bayesian networks}

\author{Robin J.\ Evans}

\begin{document}

\maketitle

\begin{abstract}
  Directed acyclic graph (DAG) models, also called Bayesian networks,
  impose conditional independence constraints on a multivariate
  probability distribution, and are widely used in probabilistic
  reasoning, machine learning and causal inference.  If latent
  variables are included in such a model, then the set of possible
  marginal distributions over the remaining (observed) variables is
  generally complex, and not represented by any DAG.  Larger classes
  of mixed graphical models, which use multiple edge types, have been
  introduced to overcome this; however, these classes do not represent
  all the models which can arise as margins of DAGs.  In this paper we
  show that this is because ordinary mixed graphs are fundamentally
  insufficiently rich to capture the variety of marginal models.

  We introduce a new class of hyper-graphs, called mDAGs, and a latent
  projection operation to obtain an mDAG from the margin of a DAG.  We
  show that each distinct marginal of a DAG model is represented by at
  least one mDAG, and provide graphical results towards characterizing
  when two such marginal models are the same.  Finally we show that
  mDAGs correctly capture the marginal structure of
  causally-interpreted DAGs under interventions on the observed
  variables.
\end{abstract}

\section{Introduction}

Directed acyclic graph (DAG) models, also known as Bayesian networks,
are widely used in probabilistic reasoning, machine learning and
causal inference \citep{bishop:07, darwiche:09, pearl:09}.  Their
popularity stems from a relatively simple definition in terms of a
Markov property, a modular structure which is computationally
scalable, their nice statistical properties, and their intuitive
causal interpretations.

DAG models are not closed under marginalization, in the sense that a
margin of a joint distribution which obeys a DAG model will not
generally be faithfully represented by any DAG.  Indeed, although DAG
models that include latent variables are widely used, they induce
models over the observed variables that are extremely complicated, and
not well understood.

Various authors have developed larger classes of graphical models to
represent the result of marginalizing (and in some cases also
conditioning) in Bayesian networks.  In the context of causal models
Pearl and Verma \citep{verma:91a, pearl:92, pearl:09} introduced mixed
graphs obtained by an operation called \emph{latent projection} to
represent the models induced by marginalizing.  These have been
developed into larger classes of graphical models such as summary
graphs, MC-graphs, ancestral graphs and acyclic directed mixed graphs
(ADMGs) which are closed under marginalization from the perspective of
conditional independence constraints \citep{koster:02, richardson:02,
  richardson:03, wermuth:11}.

As has long been known, however, these models do not fully capture the
range of marginal constraints imposed by DAG models.  In this paper we
show that no class of ordinary graphs is rich enough to do so,
regardless of how many types of edge are used.  Instead we introduce
the \emph{mDAG}, a hyper-graph which extends the idea of an ADMG to
have hyper bidirected edges; an example is given in Figure
\ref{fig:mdag}.  Intuitively, each red hyper-edge represents an
exogenous latent variable whose children are the vertices joined by
the edge.

 \begin{figure}
 \begin{center}
 \begin{tikzpicture}[rv/.style={circle, draw, very thick, minimum size=6.5mm, inner sep=1mm}, node distance=20mm, >=stealth]
 \pgfsetarrows{latex-latex};
\begin{scope}
 \node[circle, minimum size=2mm, inner sep=0mm, fill=red] (U2) {};
 \node[rv] (2) at (150:1.35) {$c$};
 \node[rv] (4) at (270:1.35) {$e$};
 \node[rv] (3) at (30:1.35) {$d$};
 \node[circle, minimum size=2mm, inner sep=0mm, fill=red, yshift=-13.5mm] (U3) at (3) {};
 \node[rv] (6) at ($(U3)+(330:1.35)$) {$f$};
 \node[rv, left of=2] (1) {$a$};
 \node[rv, left of=4, xshift=5mm] (5) {$b$};
 \draw[<->, very thick, color=red] (1) -- (2);
 \draw[->, very thick, color=red] (U2) -- (2);
 \draw[->, very thick, color=red] (U2) -- (3);
 \draw[->, very thick, color=red] (U2) -- (4);
 \draw[->, very thick, color=red] (U3) -- (6);
 \draw[->, very thick, color=red] (U3) -- (3);
 \draw[->, very thick, color=red] (U3) -- (4);
 \draw[->, very thick, color=blue] (2) -- (4);
 \draw[->, very thick, color=blue] (3) -- (6);
 \draw[->, very thick, color=blue] (5) -- (4);
 \draw[->, very thick, color=blue] (3) -- (4);
 \draw[->, very thick, color=blue] (1) .. controls +(1.2,1.2) and +(-1.2,1.2) .. (3);
\end{scope}
\end{tikzpicture}
\caption{An mDAG with maximal non-trivial bidirected edges (facets) $\{a,c\}$,
  $\{c,d,e\}$ and $\{d,e,f\}$.}
\label{fig:mdag}
\end{center}
\end{figure}
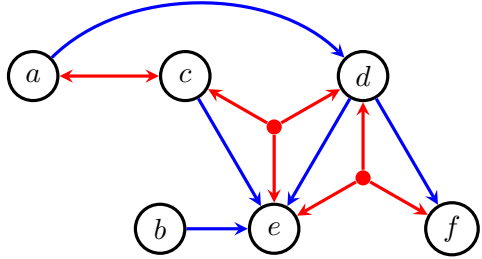

We show that mDAGs are the natural graphical object to represent
margins of DAG models.  They are rich enough to represent the variety
of models that can be induced observationally, and to graphically
represent the effect of interventions when the DAG is interpreted
causally.  In addition, if the class of possible interventions is
suitably defined, then there is a one-to-one correspondence between
causally interpreted mDAGs and the marginal models induced by causally
interpreted DAGs.  The graphical framework also provides a platform
for studying the models themselves, which are complex objects
\citep[see, for example,][]{evans:12, shpitser:14}.  We provide some
graphical results for Markov equivalence in this context, i.e.\
criteria for when two marginal models are equal, though a complete
characterization remains an open problem.

As we shall see, marginal DAG models are relatively complex and there
is, as yet, no general parameterization or fitting algorithm available
to use with them; in contrast, explicit parametric incorporation of
latent variables makes fitting relatively straightforward.  However
the latter approach has some disadvantages: most obviously it requires
additional assumptions about the nature of the latent variables that
may be implausible or untestable; additionally, the resulting models
are typically not statistically regular \citep{drton:09a}.  In
contexts where the hidden variables represent arbitrary confounders
whose nature is unknown---such as is common in epidemiological
models---it may be preferable to use a marginal DAG model rather than
an ordinary latent variable model.  For these reasons marginal DAG
models have attracted considerable interest, as the references in the
previous paragraphs attest.

The remainder of the paper is organized as follows: in Section
\ref{sec:dags} we review directed acyclic graphs and their Markov
properties; in Section \ref{sec:latent} we consider latent variables,
and discuss existing results in this area.  Section \ref{sec:mdags}
introduces mDAGs, and shows that they are rich enough to represent the
class of models induced by margins of Bayesian networks, while Section
\ref{sec:markovprop} gives Markov properties for mDAGs.  Section
\ref{sec:markov} considers Markov equivalence, and demonstrates that
ordinary mixed graphical models cannot capture the full range of
possible models.  Section \ref{sec:causal} extends the interpretation
of these models to causal settings, and Section \ref{sec:discuss}
contains a discussion including some open problems.

\section{Directed Graphical Models} \label{sec:dags}

We begin with a review of definitions concerning directed acyclic
graphs.  We omit examples of many of these ideas because these are
well known but see, for example, \citet{richardson:02} or
\citet{pearl:09} for more detail.

\begin{dfn}
A \emph{directed graph} $\D$ is a pair $(V, \mathcal{E})$, where $V$
is a finite set of \emph{vertices} and $\mathcal{E}$ a collection of
\emph{edges}, which are ordered pairs of vertices.  If $(v,w) \in
\mathcal{E}$ we write $v \rightarrow w$.  Self-loops are not allowed:
that is $(v,v) \notin \mathcal{E}$ for any $v$.  A graph is
\emph{acyclic} if it does not contain any sequences of edges of the
form $v_1 \rightarrow \cdots \rightarrow v_k \rightarrow v_1$ with $k
> 1$.  We call such a graph a \emph{directed acyclic graph} (DAG); all
the directed graphs considered in this paper are acyclic.

A \emph{path} from $v_0$ to $v_k$ is an
alternating sequence of vertices and edges $\langle v_0, e_1, v_1
\ldots, e_{k}, v_k \rangle$, such that each edge $e_i$ is between
the vertices $v_{i-1}$ and $v_i$; no repetition of vertices (or, therefore, of
edges) is permitted.  A path may contain zero edges: i.e.\ $\langle
v_0 \rangle$ is a path from $v_0$ to itself.  $v_0$ and $v_k$ are
the \emph{endpoints} of the path, and any other vertices are
\emph{non-endpoints}.
A path is
\emph{directed} from $v_0$ to $v_k$ if it is of the form $v_0
\rightarrow v_1 \rightarrow \cdots \rightarrow v_k$.

If $v \rightarrow w$ then $v$ is a \emph{parent} of $w$, and $w$ a
\emph{child} of $v$.  The set of parents of $w$ is denoted
by $\pa_\D(w)$, and the set of children of $v$ by $\ch_\D(v)$.  If there
is a directed path from $v$ to $w$ (including the case $v=w$), we say
that $v$ is an \emph{ancestor}\footnote{Note that $w$ is always an
  ancestor of itself, which differs from the convention used by some
  authors \citep[e.g.][]{lau:96}.} of $w$.  The set of ancestors of
$w$ is denoted by $\an_\D(w)$.
We apply these definitions disjunctively to sets of vertices so that
\begin{align*}
&\pa_\D(A) = \bigcup_{a \in A} \pa_\D(a), && \an_\D(A) = \bigcup_{a \in A} \an_\D(a).
\end{align*}
A set is called \emph{ancestral} if it contains all its own ancestors:
$A = \an_\D(A)$.  

Given DAGs $\D(V, \mathcal{E})$ and $\D'(V', \mathcal{E}')$, we say
that $\D'$ is a \emph{subgraph} of $\D$, and write $\D' \subseteq \D$,
if $V' \subseteq V$ and $\mathcal{E}' \subseteq \mathcal{E}$.  The
\emph{induced subgraph} of $\D$ over $A \subseteq V$ is the DAG
$\mathcal{D}_A$ with vertices $A$ and edges
$\mathcal{E}_A = \{(v,w) \in \mathcal{E} : v,w \in A\}$; that is,
those edges with both endpoints in $A$.
\end{dfn}

A \emph{graphical model} arises when a graph is identified with
structure on a multivariate probability distribution.  With each
vertex $v$ we associate a random variable $X_v$ taking values in some
set $\X_v$; the joint distribution is over the product space 
$\X_V = \times_{v \in V} \X_v$.  
In DAGs the structure takes the form of each variable
$X_v$ `depending' only upon the random variables $X_{\pa(v)}$
corresponding to its immediate parents in the graph.  Unless
explicitly stated otherwise we make no assumption about the
state-space of each of the random variables $X_v$, save that we work with
Lebesgue-Rokhlin probability spaces.  Hence $X_v$ could be discrete,
one-dimensional real, vector-valued, or a countably generated process
such as a Brownian motion \citep[see][Section 2]{rokhlin:52}.  





\begin{dfn}[Structural Equation Property] \label{dfn:sep}
  Let $\D$ be a DAG with vertices $V$, and $\X_V$ a Cartesian product
  space.  We say that a joint distribution $P$ over $\X_V$ satisfies
  the \emph{structural equation property} (SEP) for $\D$ if for some
  independent random variables $E_v$ (the \emph{error variables})
  taking values in $\mathscr{E}_v$, and measurable functions $f_v :
  \X_{\pa(v)} \times \mathscr{E}_v \rightarrow \X_v$, recursively
  setting
\begin{align*}
X_v = f_v(X_{\pa(v)}, E_v), \qquad v \in V
\end{align*}
gives $X_V$ the joint distribution $P$.  Equivalently, each $X_v$ is
$\sigma(X_{\pa(v)}, E_v)$-measurable, where $\sigma(Y)$ denotes the
$\sigma$-algebra generated by the random variable $Y$.  We denote the
collection of such distributions (the \emph{structural equation model}
for $\D$) by $\M_{se}(\D)$.
\end{dfn}

\begin{rmk} \label{rmk:arb}
The fact that we can use this recursive definition 
follows from the fact that the graph is acyclic.

Although in principle the error variables have arbitrary state-space,
it follows from the discussion in \citet[][Section 2.11]{chentsov:82}
that there is no loss of generality if they are assumed to be
uniformly distributed on $(0,1)$.
\end{rmk}

Note that the structural equation model for $\D$ does not require that
a joint density for $X_V$ exists, and in particular allows for
degenerate relationships such as functional dependence between two
variables.  If a joint density with respect to a product measure 
does exist, then the model is
equivalent to that defined by requiring the usual factorization of the
joint density \citep{pearl:09}.

\begin{rmk}
  The \emph{potential outcomes} view of causal inference
  \citep{rubin:74} considers the random function $f_v(\cdot, E_v) :
  \X_{\pa(v)} \rightarrow \X_v$, generally denoted by $X_v(\cdot) =
  f_v(\cdot, E_v)$, as the main unit of interest.  Under our
  formulation this is almost surely measurable, and we can identify
  the pair $(f_v, E_v)$ with $X_v(\cdot)$.

  In general, some care is needed when defining random functions: one
  might na\"ively choose to set, for example, $X_v(x_{\pa(v)}) \sim
  N(0,1)$ independently for each $x_{\pa(v)} \in \X_{\pa(v)}$; however
  if the indexing set $\X_{\pa(v)}$ is continuous, then the function
  $X_v(\cdot)$ will almost surely not be Lebesgue measurable, and
  therefore $X_v(X_{\pa(v)})$ is not a random variable.
\end{rmk}

The structural equation model implies that each random variable is a
measurable function of its parents in the graph; it is therefore clear
that, conditional upon its parents, each variable is independent of
the other variables already defined.  \citet{pearl:85} introduced
`d-separation' as a method for interrogating Bayesian networks about
their conditional independence implications.  The resulting Markov
property is equivalent to the structural equation property, but it is
often easier to work with in practice.

\begin{dfn}
  Let $\pi$ be a path from $v$ to $w$, and let $a$ be a non-endpoint
  on $\pi$.  We say $a$ is a \emph{collider} on the path if the two
  edges in $\pi$ which contain $a$ are both oriented towards it:
  i.e.\ $\rightarrow a \leftarrow$.  Otherwise (i.e.\ if
  $\rightarrow a \rightarrow$; $\leftarrow a \rightarrow$; or
  $\leftarrow a \leftarrow$) we say $a$ is a \emph{non-collider}.
\end{dfn}

\begin{dfn}[d-separation]
  Let $\pi$ be a path from $a$ to $b$ in a DAG $\D$; we say that $\pi$
  is \emph{blocked} by a (possibly empty) set
  $C \subseteq V \setminus \{a,b\}$ if either (i) there is a
  non-collider on $\pi$ which is also in $C$, or (ii) there is a
  collider on the path which is \emph{not} contained in $\an_\D(C)$.

 Sets $A$ and $B$ are said to be \emph{d-separated} given $C$ if all
 paths from any $a \in A$ to any $b \in B$ are blocked by $C$.
\end{dfn}

\begin{dfn}[Global Markov Property]
Let $\D$ be a DAG and $X_V$ random variables under a joint probability measure 
$P$.  We say that $P$ obeys the \emph{global Markov property} for $\D$ if 
\begin{align*}
X_A \indep X_B \,|\, X_C \, [P]
\end{align*}
whenever $A$ and $B$ are d-separated by $C$ in $\D$.  Denote the
collection of probability measures that satisfy the global Markov
property by $\M_g(\D)$.
\end{dfn}

In fact $\M_g(\D) = \M_{se}(\D)$, so the structural equation property
and the global Markov property are equivalent \citep{lauritzen:90}.  
We use $\M(\D)$ to denote these equivalent models.

\section{Latent Variables} \label{sec:latent}

In a great many practical statistical applications it is necessary to
include unmeasured random variables in a model to correctly capture
the dependence structure among observed variables.  Consider a DAG
$\D$ with vertices $V \dot\cup U$, and suppose that $(X_V, X_U) \sim P
\in \M(\D)$ (here and throughout $\dot\cup$ represents a union of
disjoint sets).  What restrictions does this place on the
\emph{marginal distribution} of $X_V$ under $P$?  In this context we
call $V$ the \emph{observed} vertices, and $X_V$ the observed
variables; similarly $U$ (respectively $X_U$) are the
\emph{unobserved} or \emph{latent} vertices (variables).

\begin{dfn}
  Let $\D$ be a DAG with vertices $V \dot\cup U$, and $\X_V$ a
  state-space for $V$.  Define the \emph{marginal DAG model} $\M(\D,
  V)$ by the collection of probability distributions $P$ over $\X_V$
  such that there exist 
\begin{enumerate}[(i)]
\item some state-space $\X_U$ for $X_U$; and 
\item a probability measure $Q \in \M(\D)$ over $\X_V \times \X_U$;
\end{enumerate} 
and $P$ is the marginal distribution of $Q$ over $X_V$.
\end{dfn}

In other words, we need to construct $(X_U, X_V)$ with joint
distribution $Q \in \M(\D)$ such that $X_V \sim P$.  Trivially, if
$U = \emptyset$ then everything is observed and $\M(\D, V) = \M(\D)$.
The problem of interest is to characterize the set $\M(\D, V)$ in
general.

\begin{rmk}
  Note that we allow the state-space of the latent variables to be
  arbitrary in principle (though see Remark \ref{rmk:arb}) and the
  model is non-parametric.  Typical latent variable models
  either assume a fixed finite number of levels for the latents, or 
  invoke some other parametric structure such as Gaussianity.  
  Such models are useful in many contexts, but have various 
  disadvantages if the aim is to remain agnostic as to the precise 
  nature of the unobserved variables.  
  In general any latent variable model will be a sub-model of the
  marginal DAG model, and may impose additional
  constraints on the observed joint distribution
  \citep[see, for example,][]{allman:13}.  This is clearly
  undesirable if it is simply an artefact of an arbitrary and untested 
  parametric structure
  applied to unmeasured variables.  In addition, latent variable
  models are often not regular and may have poor statistical
  properties, such as non-standard asymptotics \citep{drton:09a}.
  The regularity of marginal DAG models has not been established in general,
  but is known in some special cases \citep{evans:complete}.
\end{rmk}

The following proposition shows that taking margins with respect to
ancestral sets preserves the structure of the original graph,
representing an important special case.  The result is well known, see
for example \citet{richardson:02}.

\begin{prop} \label{prop:anc}
Let $\D$ and $\D'$ be DAGs with the same vertex set $V$.
\begin{enumerate}[(a)]
\item If $A \subseteq V$ is an ancestral set in $\D$, then $\M(\D, A) = \M(\D_A)$.
\item If $\D' \subseteq \D$, then $\M(\D') \subseteq \M(\D)$.
\end{enumerate}
\end{prop}

\begin{proof}
  These both follow directly from the definition of the structural
  equation property, since each variable depends only upon its
  parents.  For the first claim it is clear from the recursive form of
  the SEP that the restrictions on $X_A$ are identical for $\D$ and
  $\D_A$ if $A$ is ancestral.

  For the second claim, note that since $\pa_{\D'}(w) \subseteq
  \pa_\D(w)$, any $\sigma(X_{\pa_{\D'}(w)}, E_w)$-measurable random
  variable must also be $\sigma(X_{\pa_{\D}(w)}, E_w)$-measurable.
\end{proof}

 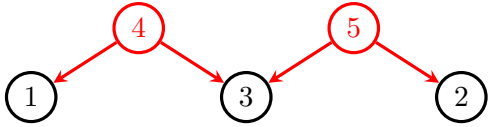
\begin{figure}
 \begin{center}
 \begin{tikzpicture}[rv/.style={circle, draw, very thick, minimum size=6.5mm, inner sep=0.5mm}, node distance=20mm, >=stealth]
 \pgfsetarrows{latex-latex};
\begin{scope}
 \node[rv] (v3) {$3$};
 \node[rv, above left of=v3, color=red, yshift=-5mm] (h1) {$4$};
 \node[rv, above right of=v3, color=red, yshift=-5mm] (h2) {$5$};
 \node[rv, below left of=h1, yshift=5mm] (v1) {$1$};
 \node[rv, below right of=h2, yshift=5mm] (v2) {$2$};
 \draw[->, very thick, color=red] (h1) -- (v1);
 \draw[->, very thick, color=red] (h1) -- (v3);
 \draw[->, very thick, color=red] (h2) -- (v2);
 \draw[->, very thick, color=red] (h2) -- (v3);
\end{scope}
  \end{tikzpicture}
 \caption{A DAG ${\mathcal{K}}$ with hidden vertices.}
\label{fig:3chain}
\end{center}
\end{figure}

\begin{exm} \label{exm:3chain}
  Consider the DAG ${\mathcal{K}}$ shown in Figure
  \ref{fig:3chain}, which contains five vertices.  We claim that
  the model defined by the margin of this graph over the vertices
  $\{1,2,3\}$ is precisely those distributions for which $X_1 \indep
  X_2$.
To see this, first note that from the global Markov property for
${\mathcal{K}}$, any distribution in $\M({\mathcal{K}},
\{1,2,3\})$ must satisfy $X_1 \indep X_2$.

Conversely, suppose that $P$ is a distribution on $(X_1,X_2,X_3)$ 
such that $X_1 \indep
X_2$; now let $(X_4, X_5, X_3) \sim P$ so that $X_4 \indep X_5$; by
Proposition \ref{prop:anc}(a) and the global Markov property 
the distribution of $(X_3,X_4,X_5)$
satisfies the Markov property for the ancestral subgraph $4
\rightarrow 3 \leftarrow 5$.  Setting $X_1 = X_4$ and $X_2 = X_5$ is
consistent with the structural equation property for
${\mathcal{K}}$, so it follows that the joint distribution of
$(X_{1}, \ldots, X_5)$ is contained in $\M({\mathcal{K}})$, and
that $(X_1, X_2, X_3) \sim P$.  Hence $P \in \M({\mathcal{K}},
\{1,2,3\})$.
\end{exm}

Even in small problems, explicitly characterizing the margin
of a DAG model can be quite tricky, as the following example shows.

\begin{exm} \label{exm:verma} Consider the DAG $\D$ in Figure
  \ref{fig:verma}, and the marginal model $\M(\D, \{1,2,3,4\})$.  By
  applying the global Markov property to $\D$, one can see that any
  joint distribution satisfies $X_1 \indep X_3 \,|\, X_2$, so this
  also holds for any marginal distribution.  It was also shown by
  \citet{robins:86} that any such distribution with a positive
  probability density must also satisfy a non-parametric constraint
  that the quantity
\begin{equation}
q(x_3, x_4) \equiv \int p_2(x_2 \,|\, x_1) \cdot p_4(x_4 \,|\, x_1, x_2, x_3) \, dx_2 \label{eqn:vermacons}
\end{equation}
is independent of $x_1$ (here $p_2$ and $p_4$ represent the relevant
conditional densities).  This does not correspond to an ordinary
conditional independence, and is known as a \emph{Verma constraint}
after \citet{verma:90} who introduced it to the computer science
literature.

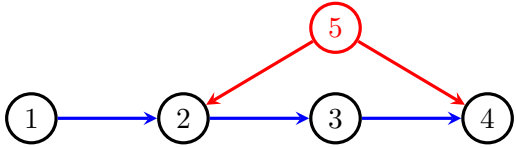
\begin{figure}
 \begin{center}
 \begin{tikzpicture}
 [rv/.style={circle, draw, very thick, minimum size=6.5mm, inner sep=0.75mm}, node distance=20mm, >=stealth]
 \pgfsetarrows{latex-latex};
 \node[rv]  (1)              {$1$};
 \node[rv, right of=1] (2) {$2$};
 \node[rv, right of=2] (3) {$3$};
 \node[rv, right of=3] (4) {$4$};
 \node at (3) [rv, color=red, yshift=12mm] (U) {$5$};
 \draw[->, very thick, color=blue] (1) -- (2);
 \draw[->, very thick, color=blue] (2) -- (3);
 \draw[->, very thick, color=blue] (3) -- (4);
 \draw[<-, very thick, color=red] (2) -- (U);
 \draw[->, very thick, color=red] (U) -- (4);
 \end{tikzpicture}
 \caption{A directed acyclic graph on five vertices.}
 \label{fig:verma}
 \end{center}
\end{figure}
\end{exm}

\subsection{Existing Results}

Margins of DAG models are of considerable interest because of their
relationship to causal models under confounding, and consequently have
been well studied.  Restricting to implications of d-separation
applied to the observed variables leads to a pure conditional
independence model; this is a super-model of the marginal DAG model
(so for Example \ref{exm:verma} we would just find
$X_1 \indep X_3 \,|\, X_2$, for instance).  This class, which we refer
to as \emph{ordinary Markov models}, was the subject of the work by
\citet{richardson:03} and \citet{evans:14} \citep[see
also][]{richardson:02}.

Constraints of the kind given in Example \ref{exm:verma} can be
generalized via the algorithm of \citet{tian:02}, and when used to
augment the ordinary Markov model yield \emph{nested Markov models}
\citep{shpitser:14}; these models are defined in Section
\ref{sec:markovprop}.  For discrete variables both ordinary and nested
Markov models are curved exponential families, and can be
parameterized and fitted using the methods of \citet{evans:10,
  evans:14}; see also \citet{shpitser:13}.  \citet{evans:complete}
shows that, up to inequality constraints, nested models are the same
as marginal DAG models when the observed variables are
discrete\footnote{In algebraic language, the marginal and nested
  models have the same Zariski closure.}: so, for example, the model
in Example \ref{exm:verma} has no equality constraints beyond the
conditional independence and (\ref{eqn:vermacons}).

In addition to conditional independences and Verma constraints,
margins also exhibit inequality constraints.  These were first
identified by \citet{bell:64}, and the earliest example in the context
of graphical models was the instrumental inequality of
\citet{pearl:95}.  \citet{evans:12} extended Pearl's work to general
DAG models and gave a graphical criterion similar to d-separation for
detecting inequality constraints.  Further inequalities are given in
\citet{fritz:12}.  \citet{bonet:01} showed that a full derivation of
inequalities in these models is likely to be very complicated in
general.  An alternative approach using information theory, also for
discrete variables, is given by \citet{chaves:14}.

A related problem to the one we consider here arises when observed and
latent variables are assumed to be jointly Gaussian.  Again one can
define an `ordinary model' using conditional independence constraints,
which is larger than the marginal model but can be smoothly
parameterized using the results in \citet{richardson:02}.  However
margins of these models also induce Verma constraints and
inequalities, as well as more exotic constraints \citep[see 8.3.1 of
][]{richardson:02}; an overview is given in \citet{drton:12}.
\citet{fox:14} characterize these models in a fairly large class of
graphs, though the general case remains an open problem.

\subsection{Reduction} \label{sec:reduce}

It might seem that to characterize general models of the form $\M(\D,
V)$ we will have to consider an infinite collection of models with
arbitrarily many latent variables, making the problem extremely hard.
However the three results in this subsection show that without any
loss of generality we can assume latent variables to be exogenous
(that is, they have no parents), and that for a fixed number of
observed variables, the number of latent variables can be limited to a
finite value.  This is in the spirit of the latent projection
operation used in \citet{pearl:09}.

\begin{dfn}
  Let $\D$ be a DAG containing a vertex $u$.  Define the
  \emph{exogenized DAG} $\mathfrak{r}(\D, u)$ as follows: take the
  vertices and edges of $\D$, and then (i) add an edge $l \rightarrow
  k$ from every $l \in \pa_\D(u)$ to $k \in \ch_\D(u)$ (if necessary),
  and (ii) delete any edge $l \rightarrow u$ for $l \in \pa_\D(u)$.
  All other edges are as in $\D$.
\end{dfn}

In other words, we join all parents of $u$ to all children of $u$ with
directed edges, and then remove edges between $u$ and its parents; the
process is most easily understood visually: see the example in Figure
\ref{fig:simp1}.  Note that if $u$ has no parents in $\D$, then
$\mathfrak{r}(\D, u) = \D$.

\begin{figure}
\begin{center}
\begin{tikzpicture}
 [rv/.style={circle, draw, thick, minimum size=6mm, inner sep=0.5mm}, node distance=14mm, >=stealth]
 \small
 \pgfsetarrows{latex-latex};
 \begin{scope}
 \node[rv] (1) {$l_1$};
 \node[rv, right of=1] (2) {$l_2$};
 \node[rv, color=red, below of=1, xshift=7mm] (U) {$u$};
 \node[rv, below of=U] (4) {$k_2$};
 \node[rv, left of=4] (3) {$k_1$};
 \node[rv, right of=4] (5) {$k_3$};
 \draw[<-, very thick, color=red] (U) -- (1);
 \draw[->, very thick, color=blue] (2) -- (5);
 \draw[<-, very thick, color=red] (U) -- (2);
 \draw[->, very thick, color=red] (U) -- (3);
 \draw[->, very thick, color=red] (U) -- (4);
 \draw[->, very thick, color=red] (U) -- (5);
 \node[below of=4, yshift=-3mm] {(a)};
  \end{scope}
 \begin{scope}[xshift=6cm]
 \node[rv] (1) {$l_1$};
 \node[rv, right of=1] (2) {$l_2$};
 \node[below of=1, xshift=7mm] (U) {};
 \node[rv, below of=U] (4) {$k_2$};
 \node[rv, left of=4] (3) {$k_1$};
 \node[rv, right of=4] (5) {$k_3$};
 \node[rv, color=red, below of=4, yshift=4mm] (W) {$u$};
 \draw[->, very thick, color=blue] (1) -- (3);
 \draw[->, very thick, color=blue] (2) -- (3);
 \draw[->, very thick, color=blue] (1) -- (4);
 \draw[->, very thick, color=blue] (2) -- (4);
 \draw[->, very thick, color=blue] (1) -- (5);
 \draw[->, very thick, color=blue] (2) -- (5);
 \draw[->, very thick, color=red] (W) -- (3);
 \draw[->, very thick, color=red] (W) -- (4);
 \draw[->, very thick, color=red] (W) -- (5);
 \node[below of=W, yshift=7mm] {(b)};
 \end{scope}
 \end{tikzpicture}
 \caption{(a) A DAG, $\D$, and (b) the exogenized version
   $\mathfrak{r}(\D, u)$.  The two DAGs induce the same marginal model
   over the vertices $\{l_1,l_2,k_1,k_2,k_3\}$.}
 \label{fig:simp1}
\end{center}
\end{figure}
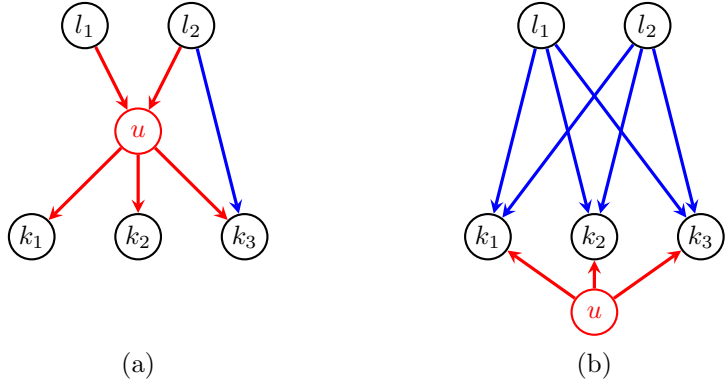

\begin{lem} \label{lem:simp1} Let $\D$ be a DAG with vertices $V \dot\cup
  \{u\}$, and $\tilde\D \equiv \mathfrak{r}(\D, u)$.  Then $\M(\D, V)
  = \M(\tilde\D, V)$; i.e.\ the marginal models induced by the two
  graphs over $V$ are the same.
\end{lem}

\begin{proof}
  If $u$ has no parents in $\D$ then the result is trivial, since $\D
  = \tilde\D$.  Otherwise let $L = \pa_{\D}(u)$ and $K = \ch_{\D}(u)$.
  Suppose $P \in \M(\D, V)$, so one can construct $(X_u, X_V) \sim Q
  \in \M(\D)$ such that $X_V \sim P$.  Let $Q$ be generated using the SEP by
  independent error variables $(E_v: v \in V \cup \{u\})$, so that each
  $X_v$ is $\sigma(X_{\pa_\D(v)}, E_v)$-measurable.

  Now let $\tilde{X}_u = E_u$, and all other $X_v$ remain unchanged,
  so that $\tilde{X}_u$ is $\sigma(E_u)$-measurable.  The only other
  variables whose parents sets are different in $\tilde{\D}$ are those in
  $K$, so we need only show that $X_k$ is $\sigma(\tilde{X}_u, X_L,
  X_{\pa_\D(k)}, E_k)$-measurable for $k \in K$.  Since $X_u$ is
  $\sigma(X_L, E_u) = \sigma(X_L, \tilde{X}_u)$-measurable, it follows
  that
\[
\sigma(X_u, X_{\pa_\D(k)}, E_k) \subseteq \sigma(\tilde{X}_u, X_L, X_{\pa_\D(k)}, E_k).
\]
$X_k$ is $\sigma(X_u, X_{\pa_\D(k)}, E_k)$-measurable by the
definition of $\M(\D)$, so it is also $\sigma(\tilde{X}_u, X_L,
X_{\pa_\D(k)}, E_k)$-measurable.  The joint distribution $\tilde{Q}$
of $(\tilde{X}_u, X_V)$ is therefore contained in $\M(\tilde{\D})$,
and so $P \in \M(\tilde{\D}, V)$.

Conversely, if $(\tilde{X}_u, X_V) \sim \tilde{Q} \in \M(\tilde{\D})$,
let $E_u = \tilde{X}_u$, and $X_u = (X_L, \tilde{X}_u)$; then $E_u$ is
independent of other error variables, and $X_u$ is $\sigma(X_L,
E_u)$-measurable.  For $k \in K$,
\begin{align*}
  \sigma(X_u, X_{\pa_\D(k)}, E_k) \supseteq \sigma(\tilde{X}_u, X_L,
  X_{\pa_\D(k)}, E_k),
\end{align*} 
so $(X_u, X_V) \sim Q \in \M(\D)$.
\end{proof}

As a consequence of this lemma it is sufficient to consider models in
which the unobserved vertices are exogenous.  Our second result shows
that only a finite number of exogenous latent variables are
necessary.

\begin{lem} \label{lem:simp2} 
  Let $\D$ be a DAG with vertices $V \dot\cup \{ u, w\}$ (where $u \neq
  w$), such that $\pa_\D(w) = \pa_\D(u) = \emptyset$ and $\ch_\D(w)
  \subseteq \ch_\D(u)$.  Then $\M(\D, V) = \M(\D_{-w}, V)$, where
  $\D_{-w}$ is the induced subgraph of $\D$ after removing $w$.
\end{lem}

\begin{proof}
  By Proposition \ref{prop:anc}(b), $\M(\D_{-w}, V) \subseteq \M(\D,
  V)$.  Take $P \in \M(\D, V)$, so that there exists $(X_V, X_u, X_w)
  \sim Q \in \M(\D)$ whose $V$-margin is $P$.  Letting $\tilde{X}_u =
  (X_u, X_w)$ note that $(X_V, \tilde{X}_u)$ satisfies the SEP for 
$\D_{-w}$.
Hence $P \in \M(\D_{-w}, V)$.
\end{proof}

This result, combined with Lemma \ref{lem:simp1}, shows that for a
fixed set of observed variables $V$, there are only finitely many
distinct models of the form $\M(\D, V)$.  In particular, all
unobserved vertices may be assumed to be exogenous, and their child
sets to correspond to maximal sets of observed vertices.  An example
of two DAGs shown to have equal marginal models by this result is
given in Figure \ref{fig:simp2}.

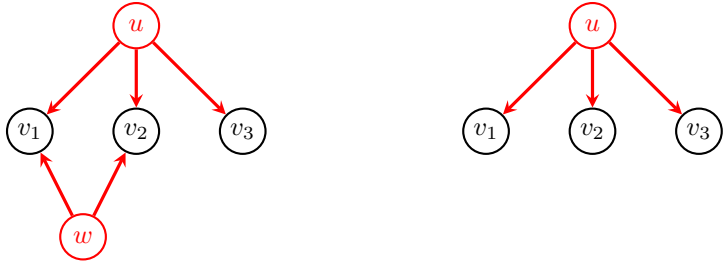
\begin{figure}
\begin{center}
\begin{tikzpicture}
 [rv/.style={circle, draw, thick, minimum size=6mm, inner sep=0.5mm}, node distance=14mm, >=stealth]
 \small
 \pgfsetarrows{latex-latex};
 \begin{scope}
 \node[rv] (1) {$v_1$};
 \node[rv, right of=1] (2) {$v_2$};
 \node[rv, right of=2] (3) {$v_3$};
 \node[rv, color=red, above of=2] (U) {$u$};
 \node[rv, color=red, below of=2, xshift=-7mm] (W) {$w$};
 \draw[->, very thick, color=red] (U) -- (1);
 \draw[->, very thick, color=red] (U) -- (2);
 \draw[->, very thick, color=red] (U) -- (3);
 \draw[->, very thick, color=red] (W) -- (1);
 \draw[->, very thick, color=red] (W) -- (2);
 \end{scope}
 \begin{scope}[xshift=60mm]
 \node[rv] (1) {$v_1$};
 \node[rv, right of=1] (2) {$v_2$};
 \node[rv, right of=2] (3) {$v_3$};
 \node[rv, color=red, above of=2] (U) {$u$};
 \draw[->, very thick, color=red] (U) -- (1);
 \draw[->, very thick, color=red] (U) -- (2);
 \draw[->, very thick, color=red] (U) -- (3);
 \end{scope}
 \end{tikzpicture}

 \caption{Two DAGs whose marginal models over the vertices $\{v_1,v_2,v_3\}$ are the same.}
  \label{fig:simp2}
  \end{center}
\end{figure}

We can make one final simplification, again without any loss of
generality.

\begin{lem} \label{lem:simp3} 
  Let $\D$ be a DAG with vertices $V \dot\cup \{u\}$, such that $u$ has no
  parents and at most one child.  Then $\M(\D, V) = \M(\D_{-u}, V)$.
\end{lem}

\begin{proof}
  $\M(\D_{-u}, V) \subseteq \M(\D, V)$, so suppose $P \in \M(\D, V)$.
  For the unique $v \in \ch_\D(u)$ (if indeed there is any such $v$),
  let $\tilde E_v = (E_v, E_u)$, so $\tilde E_v \indep (E_w : w \in
  V)$, and $X_v$ is $\sigma(X_{\pa(v)}, E_v) = \sigma(X_{\pa(v)
    \setminus u}, \tilde{E}_v)$-measurable.  Then $P \in \M(\D, V)$.
\end{proof}

The combination of these results means that we can restrict our
attention to models in which the latent variables are exogenous, and
have non-nested sets of children of size at least two.  A similar
conclusion is reached by \citet{pearl:92}, but the authors also claim
that each latent variable can be assumed to have \emph{exactly} two
children.  In the context of models of conditional independence this
is correct, but in general it is too restrictive, as we show in
Section \ref{sec:admgs}.

\section{mDAGs} \label{sec:mdags}

The results of the previous section suggest a way to construct a new
class of graph, rich enough to represent the distinct models that can
arise as the margins of DAGs.  First we define the following abstract
object, which will be used to represent latent structure.

\begin{dfn}
A \emph{simplicial complex} (or abstract simplicial complex),
$\mathcal{B}$, over a finite set $V$ is a collection of non-empty
subsets of $V$ such that
\benum[(i)]
\item $\{v\} \in \mathcal{B}$ for all $v \in V$;
\item for non-empty sets $A \subseteq B \subseteq V$ we have $B \in
  \mathcal{B} \implies A \in \mathcal{B}$.  
\eenum
The inclusion maximal elements of $\mathcal{B}$ are called
\emph{facets}.  Any simplicial complex $\mathcal{B}$ can be
characterized by its non-trivial facets (i.e.\ those of size at least
2), denoted by $\bar{\mathcal{B}}$.
\end{dfn}

\begin{dfn}
  An \emph{mDAG} (marginalized DAG) $\G$ is a triple $(V, \mathcal{E},
  \mathcal{B})$, where $(V, \mathcal{E})$ defines a DAG, and
  $\mathcal{B}$ is an abstract simplicial complex on $V$.  The
  elements of $\mathcal{B}$ are called the \emph{bidirected faces}.
\end{dfn}

DAGs correspond to mDAGs whose bidirected faces are just singleton
vertices: $\mathcal{B} = \{\{v\} : v \in V\}$.  We can represent an
mDAG as a graph with ordinary directed edges $\mathcal{E}$, and
bidirected hyper-edges corresponding to the non-trivial facets $\bar{\mathcal{B}}$.  We call $(V, \mathcal{E})$ the \emph{underlying DAG},
and draw its edges in blue; the bidirected hyper-edges are in red.
See the example in Figure \ref{fig:mdag}.  If $w$ has no parents and
$\{w\}$ is a facet of $\mathcal{B}$, we say that $w$ is
\emph{exogenous}.  Informally we may think of each facet $B$ as
representing a latent variable with children $B$.  The definitions of
parents, children, ancestors and ancestral sets are extended to mDAGs
by applying them to the underlying DAG, ignoring the bidirected faces.

Visually, there is some resemblance between the bidirected hyper-edges
in mDAGs and the factor nodes in factor graphs, but this similarity is
only superficial: for example, factor graphs do not require inclusion
maximality \citep{kschischang:01}.

If we restrict the facets of $\mathcal{B}$ to have size at most 2 (so
that $\mathcal{B}$ is an `edge complex'), then the definition of an
mDAG is isomorphic to that of an \emph{acyclic directed mixed graph}
or ADMG \citep{richardson:03}.  Clearly then, mDAGs are a richer class
of graphs: the relationship between mDAGs and ADMGs is explained further
in Section \ref{sec:admgs}.





\begin{dfn}[Subgraph]
  Let $\G(V, \mathcal{E}, \mathcal{B})$ and $\H(V', \mathcal{E}',
  \mathcal{B}')$ be mDAGs.  Say that $\H$ is a \emph{subgraph} of
  $\G$, and write $\H \subseteq \G$, if $V' \subseteq V$,
  $\mathcal{E}' \subseteq \mathcal{E}$, and $\mathcal{B}' \subseteq
  \mathcal{B}$.

  The \emph{induced subgraph} of $\G$ over $A \subseteq V$ is the mDAG
  defined by the induced underlying DAG $(A, \mathcal{E}_A)$ and
  bidirected faces $\mathcal{B}_A = \{B \subseteq A : B \in \mathcal{B}\}$.
  In other words, taking those parts of each edge which intersect with
  the vertices in $A$.
\end{dfn}

\subsection{Latent Projection}

We now relate margins of DAG to mDAGs, via an operation called latent
projection.  This is based on the approach taken by \citet{pearl:09},
but allows for joint dependence of more than two variables due to a
common `cause' or ancestor.

\begin{dfn}
  Let $\G$ be an mDAG with bidirected faces $\mathcal{B}$, and let $W,
  U$ be disjoint sets of vertices in $\G$.  We say that the vertices
  in $W$ share a \emph{hidden common cause} in $\G$, with respect to
  $U$, if there exists a set $B \in \mathcal{B}$ such that
\begin{enumerate}[(i)]
\item $B \subseteq U \dot\cup W$; and
\item for each $w \in W$ there is a directed path $\pi_w$ from some $b
  \in B$ to $w$, with all vertices on $\pi_b$ being in $U \cup \{w\}$.
\end{enumerate}
\end{dfn}

If $\G$ is a DAG, a hidden common cause is a common ancestor $a \in V$
of each $w \in W$, where $a$ and the vertices on a directed path
between $a$ and $w$ are unobserved.  Note that if $W \in \mathcal{B}$
then $W$ is trivially a hidden common cause with respect to any
$U \subseteq V \setminus W$.

The concept of a hidden common cause is similar to a system of
\emph{treks} which induce latent correlation; see, for example,
\citet{foygel:12}.  The difference is that treks only consider
pairwise dependence, not dependence between an arbitrary collection of
variables.


\begin{exm}
  Let $\G$ be the DAG in Figure \ref{fig:project}(a).  The vertices
  $W=\{3, 4, 5, 6\}$ share a hidden common cause $B=\{1\}$ with
  respect to $U=\{1,2\}$.  In the mDAG in Figure \ref{fig:project}(c)
  the set of vertices $W=\{3,4,5,6\}$ share a hidden common cause in
  the bidirected facet $\{2,3,4\}$, with respect to $\{2\}$.
\end{exm}

 \begin{figure}
 \begin{center}
 \begin{tikzpicture}[rv/.style={circle, draw, very thick, minimum size=6.5mm, inner sep=1mm}, node distance=20mm, >=stealth]
 \pgfsetarrows{latex-latex};
\begin{scope}[yshift=4cm]
 \node[rv] (U2) {1};
 \node[rv] (1) at (180:1.5) {2};
 \node[rv] (2) at (300:1.35) {3};
 \node[rv] (3) at (60:1.35) {4};
 \node[rv, left of=1, xshift=10mm, yshift=12mm] (4) {5};
 \node[rv, left of=1, xshift=10mm, yshift=-12mm] (5) {6};
 \node[rv, left of=1, xshift=0mm] (6) {7};
 \draw[->, very thick, color=blue] (1) -- (5);
 \draw[->, very thick, color=blue] (1) -- (4);
 \draw[->, very thick, color=blue] (U2) -- (3);
 \draw[->, very thick, color=blue] (U2) -- (2);
 \draw[->, very thick, color=blue] (U2) -- (1);
 \draw[->, very thick, color=blue] (6) -- (1);
\node[below of=1, yshift=0mm] {(a)};
\end{scope}
\begin{scope}[xshift=6cm, yshift=4cm]
 \node[rv] (1) {1};
 \node[rv] (2) at (315:1.35) {3};
 \node[rv] (3) at (45:1.35) {4};
 \node[rv] (4) at (135:1.35) {5};
 \node[rv] (5) at (225:1.35) {6};
 \node[rv, left of=1, xshift=0mm] (6) {7};
 \draw[<->, very thick, color=red] (4) -- (5);
 \draw[->, very thick, color=blue] (1) -- (5);
 \draw[->, very thick, color=blue] (1) -- (4);
 \draw[->, very thick, color=blue] (1) -- (3);
 \draw[->, very thick, color=blue] (1) -- (2);
 \draw[->, very thick, color=blue] (6) -- (4);
 \draw[->, very thick, color=blue] (6) -- (5);
\node[below of=1, yshift=0mm] {(b)};
\end{scope}
\begin{scope}
 \node[circle, minimum size=2mm, inner sep=0mm, fill=red] (U2) {};
 \node[rv] (1) at (180:1.35) {2};
 \node[rv] (2) at (300:1.35) {3};
 \node[rv] (3) at (60:1.35) {4};
 \node[rv, left of=1, xshift=10mm, yshift=12mm] (4) {5};
 \node[rv, left of=1, xshift=10mm, yshift=-12mm] (5) {6};
 \node[rv, left of=1, xshift=0mm] (6) {7};
 \draw[->, very thick, color=blue] (1) -- (5);
 \draw[->, very thick, color=blue] (1) -- (4);
 \draw[->, very thick, color=red] (U2) -- (3);
 \draw[->, very thick, color=red] (U2) -- (2);
 \draw[->, very thick, color=red] (U2) -- (1);
 \draw[->, very thick, color=blue] (6) -- (1);
\node[below of=1, yshift=0mm] {(c)};
\end{scope}
\begin{scope}[xshift=6cm]
 \node[circle, minimum size=2mm, inner sep=0mm, fill=red] (U2) {};
 \node[rv] (4) at (135:1.35) {5};
 \node[rv] (5) at (225:1.35) {6};
 \node[rv] (2) at (315:1.35) {3};
 \node[rv] (3) at (45:1.35) {4};
 \node[rv] (6) at (180:2) {7};
 \draw[->, very thick, color=red] (U2) -- (5);
 \draw[->, very thick, color=red] (U2) -- (4);
 \draw[->, very thick, color=red] (U2) -- (2);
 \draw[->, very thick, color=red] (U2) -- (3);
 \draw[->, very thick, color=blue] (6) -- (4);
 \draw[->, very thick, color=blue] (6) -- (5);
\node[below of=U2, yshift=0mm] {(d)};
\end{scope}
 \end{tikzpicture}
 \caption{(a) A DAG on seven vertices, and (b) its latent projection
   to an mDAG over $\{1,3,4,5,6,7\}$, (c) over $\{2,3,4,5,6,7\}$ and
   (d) over $\{3,4,5,6,7\}$.}
 \label{fig:project}
  \end{center}
\end{figure}
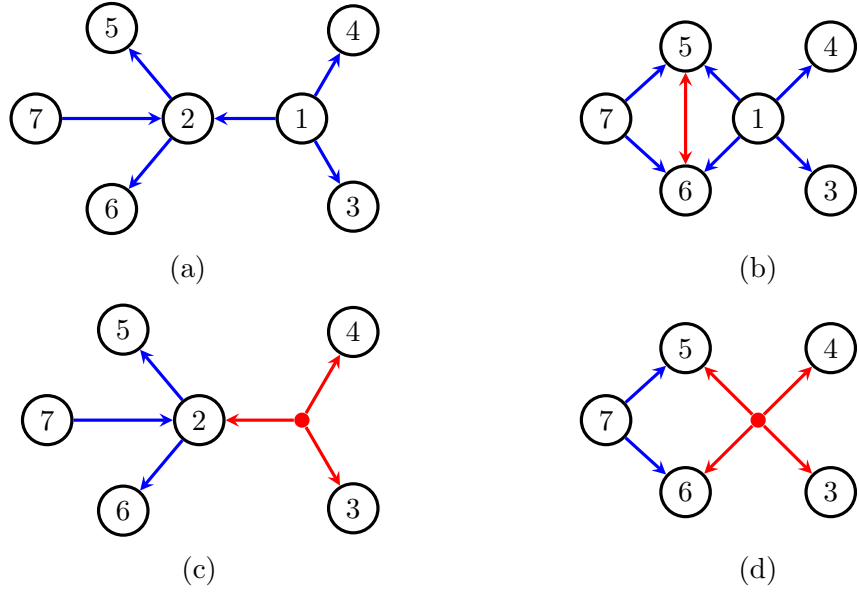

The hidden common cause forms the basis for determining which vertices
should share a bidirected face in an mDAG after projecting out some of
the variables.  We formalize this with the next definition.

\begin{dfn}
  Let $\G$ be an mDAG with vertices $V \dot\cup U$.  The \emph{latent
    projection} of $\G$ onto $V$, denoted by $\mathfrak{p}(\G, V)$, is an
  mDAG with vertices $V$, and edges $\mathcal{E}'$ and bidirected faces $\mathcal{B}'$
  defined as follows:
\begin{itemize}
\item $(a,b) \in \mathcal{E}'$ whenever $a \neq b$ and there is a
  directed path $a \rightarrow \cdots \rightarrow b$ in $\G$, with all
  non-endpoints in $U$;
\item $W \in \mathcal{B}'$ whenever the vertices $W \subseteq V$ share
  a hidden common cause in $\G$ with respect to $U$.
\end{itemize}
\end{dfn}

It is straightforward to see that $\mathcal{B}'$ is an abstract
simplicial complex, and therefore the definition above gives an mDAG.

\begin{exm}
  Consider the mDAG in Figure \ref{fig:project}(a), and its latent
  projection after projecting out the vertex $2$, shown in Figure
  \ref{fig:project}(b).  In the original graph the directed paths $7
  \rightarrow 2 \rightarrow 5$ and $7 \rightarrow 2 \rightarrow 6$ are
  manifested as the directed edges $7 \rightarrow 5$ and $7
  \rightarrow 6$ in the projection.  Additionally, there is a hidden
  common cause for the vertices $5,6$ (as noted in the previous
  example), so we end up with a bidirected facet $\{5,6\}$ in the
  projection.
  The projection of the graph in Figure \ref{fig:project}(b) onto
  $\{3,4,5,6,7\}$ is shown in (d).
\end{exm}

\begin{dfn}
Let $\G(V,\mathcal{E},\mathcal{B})$ be an mDAG with bidirected facets
$\bar{\mathcal{B}}$.  We define $\bar{\G}$, the \emph{canonical DAG}
associated with $\G$, as the DAG with vertices $V \cup
\bar{\mathcal{B}}$ and edges
\[
\mathcal{E} \cup \{B \rightarrow v \,:\, v \in B \in \bar{\mathcal{B}}\}.
\]
\end{dfn}

In other words, we replace every non-trivial facet $B \in \mathcal{B}$
with a vertex whose children are precisely the elements of $B$.  The
canonical DAG associated with the mDAG from Figure \ref{fig:mdag} is
shown in Figure \ref{fig:canon}. 

 \begin{figure}
 \begin{center}
 \begin{tikzpicture}[rv/.style={circle, draw, very thick, minimum size=6.5mm, inner sep=1mm}, node distance=25mm, >=stealth]
 \pgfsetarrows{latex-latex};
\begin{scope}[yshift=-5.5cm]
 \node[rv, color=red] (U2) {$B_2$};
 \node[rv] (2) at (150:1.7) {$c$};
 \node[rv] (4) at (270:1.7) {$e$};
 \node[rv] (3) at (30:1.7) {$d$};
 \node[rv, left of=2, xshift=-6mm] (1) {$a$};
 \node[rv, left of=2, xshift=9.5mm, color=red] (U1) {$B_1$};
 \node[rv, yshift=-17mm, color=red] (U3) at (3) {$B_3$};
 \node[rv] (6) at ($(U3)+(330:1.7)$) {$f$};
 \node[rv, left of=4, xshift=5mm] (5) {$b$};
 \draw[->, very thick, color=red] (U1) -- (2);
 \draw[->, very thick, color=red] (U1) -- (1);
 \draw[->, very thick, color=red] (U2) -- (2);
 \draw[->, very thick, color=red] (U2) -- (3);
 \draw[->, very thick, color=red] (U2) -- (4);
 \draw[->, very thick, color=red] (U3) -- (6);
 \draw[->, very thick, color=red] (U3) -- (3);
 \draw[->, very thick, color=red] (U3) -- (4);
 \draw[->, very thick, color=blue] (2) -- (4);
 \draw[->, very thick, color=blue] (5) -- (4);
 \draw[->, very thick, color=blue] (3) -- (4);
 \draw[->, very thick, color=blue] (3) -- (6);
 \draw[->, very thick, color=blue] (1) .. controls +(1.2,1.2) and +(-1.2 ,1.2) .. (3);
\end{scope}
 \end{tikzpicture}
\caption{The canonical DAG associated with the mDAG in Figure \ref{fig:mdag}.}
 \label{fig:canon}
  \end{center}
 \end{figure}
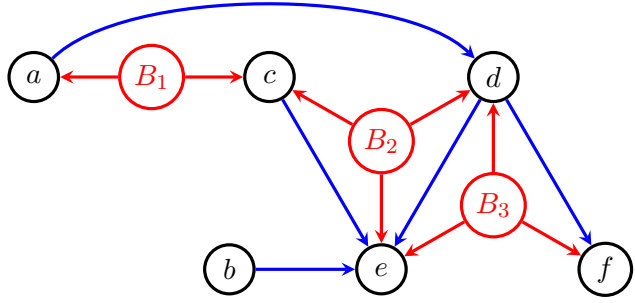

\begin{prop} \label{prop:proj}
  Let $\G$ be an mDAG with vertex set $V$.
  \begin{enumerate}[(a)]
    \item $\H \subseteq \G \implies \mathfrak{p}(\H, W) \subseteq \mathfrak{p}(\G, W)$ for any $W \subseteq V$;
    \item $\mathfrak{p}(\bar\G, V) = \G$;
    \item if $A \subseteq V$ is an ancestral set in $\G$, then $\mathfrak{p}(\G, A) = \G_A$.
  \end{enumerate}
\end{prop}

\begin{proof}
  (a): If $\H$ is a subgraph of $\G$, then any
  directed path or hidden common cause in $\H$ must also be found in
  $\G$.  

  (b): Since $\bar\G$ is a DAG on vertices $V \cup \bar{\mathcal{B}}$
  and no $B \in \bar{\mathcal{B}}$ has any parents in $\bar\G$, the
  only directed edges added in $\mathfrak{p}(\bar\G, V)$ are those
  already joining elements of $V$ in $\bar\G$, and therefore are
  precisely the directed edges in $\G$.  The only hidden common causes
  with respect to $\bar{\mathcal{B}}$ are singletons $\{v\}$ and
  subsets of any $B \in \bar{\mathcal{B}}$, whose children are all
  observed.  Hence the bidirected faces in $\mathfrak{p}(\bar\G, V)$
  are precisely $\mathcal{B}$.

  (c): Since $A$ is ancestral, any directed paths between elements of
  $A$ have all vertices in $A$, and there are no directed paths from
  $V \setminus A$ to $A$ (hence there are no hidden common causes).
\end{proof}

A critical fact about latent projection is that it does not matter
in what order we project out vertices, or indeed if we do several 
at once.

\begin{thm} \label{thm:proj}
Let $\G$ be an mDAG with vertices $V \dot\cup U_1 \dot\cup U_2$.  Then
\begin{align*}
\mathfrak{p}(\G, V) = \mathfrak{p}(\mathfrak{p}(\G, V \cup U_1), V) = \mathfrak{p}(\mathfrak{p}(\G, V \cup U_2), V).
\end{align*}
That is, the order of projection does not matter.
\end{thm}

The proof of this result is found in the Appendix.  The commutativity
is illustrated in Figure \ref{fig:project}: if we first project out
$1$ and then $2$ from the DAG (a) we obtain the mDAGs in (c) and then
(d) respectively.  If the order of projection is reversed we obtain
the mDAGs in (b) and then (d).

A second crucial fact is that if two DAGs have the same latent
projection onto a set $V$, then their marginal models over $V$ are
also the same.  To prove this we use the following two lemmas, which
show that two different DAGs result in the same mDAG if their margins
are equivalent by Lemmas \ref{lem:simp1}, \ref{lem:simp2} and
\ref{lem:simp3}.

\begin{lem} \label{lem:simpproj1} 
  Let $\D$ be a DAG with vertices $V \dot\cup \{u\}$, and
  $\mathfrak{r}(\D, u)$ the exogenized DAG for $u$.  Then
\begin{align*}
\mathfrak{p}(\D, V) = \mathfrak{p}(\mathfrak{r}(\D, u), V).
\end{align*}
\end{lem}

\begin{proof}
  From the definition of $\mathfrak{r}$, any directed paths passing
  through $u$ as an intermediate node $l \rightarrow u \rightarrow k$
  in $\D$ are replaced by $l \rightarrow k$ in
  $\mathfrak{r}(\D, u)$.  Hence the directed edges in both
  projections are the same.

  The only vertex being projected out is $u$ and since its child set
  is the same in both $\D$ and $\mathfrak{r}(\D, u)$, the groups of
  vertices sharing a hidden common cause with respect to $\{u\}$ will
  remain unchanged.  Hence the bidirected faces in both projections
  are the same.
\end{proof}

\begin{lem} \label{lem:simpproj2} 
  Let $\G$ be an mDAG with vertices $V \dot\cup U$, containing an
  exogenous vertex $w \in U$.  If either $|\ch_\G(w)| \leq 1$, or
  $\ch_\G(w) \subseteq \ch_\G(u)$ for some $u \in U$, then
\begin{align*}
\mathfrak{p}(\G, V) = \mathfrak{p}(\G_{-w}, V).
\end{align*}
\end{lem}

\begin{proof}
  Since $w$ has no parents, there are no directed paths containing it
  as an intermediate vertex; hence we need only show that if some
  vertices in $V$ share a hidden common cause in $\G$ with
  respect to $U$, then they also share one in $\G_{-w}$ with respect to
  $U \setminus \{w\}$.

  Since $w$ is exogenous this is clearly true whenever the hidden
  common cause is not $\{w\}$, and so if $w$ has no children the
  result is trivial.  If $|\ch_\G(w)| = \{k\}$ then $\{k\}$ will also
  serve as a hidden common cause.

  If $\ch_\G(w) \subseteq \ch_\G(u)$ for some $u \in U$ then clearly
  any vertices which share $\{w\}$ as a hidden common cause in $\G$
  will also have $\{u\}$ as a hidden common cause in $\G$ and $\G_{-w}$.
\end{proof}

\begin{thm} \label{thm:equiv} 
  Let $\D$, $\D'$ be two DAGs whose latent
  projections onto some set $V$ are the same.  Then $\M(\D,
  V) = \M(\D', V)$.
\end{thm}

\begin{proof}
  Let $\G = \mathfrak{p}(\D, V)$ be the latent projection.  We will
  show that $\M(\D, V) = \M(\bar{\G}, V)$, and thereby prove the
  result.  Let the vertex set of $\mathcal{D}$ be $V \dot\cup U$.

  If no vertex in $U$ has any parents in $\D$, each vertex in $U$ has
  at least two children, and their child sets are never nested, then
  $\D = \bar\G$ and there is nothing to prove.  Otherwise suppose $u
  \in U$ has at least one parent.  Then $\mathfrak{r}(\D, u)$ has the
  same latent projection onto $V$ as $\D$ by Lemma
  \ref{lem:simpproj1}, and $\M(\mathfrak{r}(\D, u), V) = \M(\D, V)$ by
  Lemma \ref{lem:simp1}.  The problem reduces to $\mathfrak{r}(\D,
  u)$, and by repeated application it reduces to DAGs in which no
  vertex in $U$ has any parents.

  Similarly, if either $w \in U$ has only one child, or $\ch_\G(w)
  \subseteq \ch_\G(u)$ for some other $u \in U$, then by Lemmas
  \ref{lem:simp2} and \ref{lem:simp3} we have $\M(\D_{-w}, V) = \M(\D,
  V)$ and by Lemma \ref{lem:simpproj2} $\mathfrak{p}(\D_{-w}, V) =
  \G$, so the problem reduces to $\D_{-w}$.  It follows that we can
  reduce to the canonical DAG $\bar{\G}$, and the result is proved.
\end{proof}

This result shows that mDAGs are rich enough to fully express the
class of marginal DAG models.  In Section \ref{sec:markov} we will see
that ordinary (i.e.\ not hyper) graphs are unable to do this, and in
Section \ref{sec:causal} that mDAGs are, from a causal perspective,
the natural object to represent these models.

\section{Markov Properties} \label{sec:markovprop}

We are now in a position to define a Markov property for mDAGs that
relates to the original problem of characterizing the margins of DAG
models.

\begin{dfn}
  Say that $P$ obeys the \emph{marginal Markov property} for an mDAG
  $\G$ with vertices $V$, if it is contained within the marginal DAG
  model of the canonical DAG: $P \in \M(\bar{\G}, V)$.  We denote the
  set of such distributions (the \emph{marginal model}) by $\M_m(\G)$.
\end{dfn}

For instance, we know from Example \ref{exm:3chain} that the marginal
model for $1 \leftrightarrow 3 \leftrightarrow 2$ is the collection of
distributions under which $X_1 \indep X_2$.

It follows from Theorem \ref{thm:equiv} that the marginal model of any
DAG $\M(\G, V)$ is the same as the model obtained by applying the
marginal Markov property to its latent projection
$\mathfrak{p}(\G, V)$.  For some $W \subseteq V$ we denote the
marginal model of an mDAG with respect to $W$ as
$\M_m(\G, W) \equiv \M(\bar\G, W)$.  Note that Theorem \ref{thm:proj}
shows that this is a sensible definition.

\begin{prop} \label{prop:subgraph}
Let $\G, \H$ be mDAGs with vertex set $V$.
\begin{enumerate}[(a)]
 \item If $A$ is an ancestral set in $\G$, then $\M_m(\G_A) = \M_m(\G, A)$.
 \item If $\H \subseteq \G$, then $\M_m(\H) \subseteq \M_m(\G)$.
\end{enumerate}
\end{prop}

\begin{proof}
  (a) By definition $\M_m(\G, A) = \M(\bar{\G}, A) =
  \M_m(\mathfrak{p}(\bar{\G}, A))$, and from Proposition \ref{prop:proj}
  $\mathfrak{p}(\bar{\G}, A) = \G_A$.  

  (b) If $\H \subseteq \G$ then $\bar{\H} \subseteq \bar\G$, so by
  Proposition \ref{prop:anc} $\M(\bar\H) \subseteq \M(\bar\G)$.  It
  follows that $\M(\bar\H, V) \subseteq \M(\bar\G, V)$, giving the
  required result.
\end{proof}

The marginal Markov property also implies certain factorizations of
the joint density, if one exists.  To describe them, we first need to
define a special subgraph.

\begin{dfn}
  Let $\G(V, \mathcal{E}, \mathcal{B})$ be an mDAG with vertices $V$.
  Say that $C \subseteq V$ is \emph{bidirected-connected} if for every
  $v,w \in C$ there is a sequence of vertices
  $v=v_0, v_1, \ldots, v_k= w$ all in $C$ such that
  $\{v_{i-1}, v_i\} \in \mathcal{B}$ for $i=1,\ldots,k$.  A maximal
  bidirected-connected set is called a \emph{district}.

  Let $\G$ be an mDAG with district $D$.  The graph $\G[D]$ is the
  mDAG with vertices $D \cup \pa_\G(D)$, directed edges
  $D \cup \pa_\G(D)$ to $D$, and bidirected edges
  $\mathcal{B}_D = \{ B \subseteq D : B \in \mathcal{B}\}$.
\end{dfn}

In other words, $\G[D]$ is the induced sub-graph over $D$, together
with any directed edges that point into $D$ (and the associated
vertices).  As an example, for the mDAG in Figure
\ref{fig:verma:mdag}(a) has districts $\{1\}$, $\{3\}$ and $\{2,4\}$.
The subgraph corresponding to $D=\{2,4\}$ is shown in Figure
\ref{fig:verma:mdag}(b).

\begin{prop} \label{prop:nested}
Let $\G$ be an mDAG with districts $D_1, \ldots, D_k$, and suppose 
that $P$ with density $p$ obeys the marginal Markov property for $\G$.
Then 
\bal
p(x_V) = \prod_{i=1}^k q_i(x_{D_i} \,|\, x_{\pa(D_i)\setminus D_i}),
\eal
for some conditional distributions $q_i$ that obey the marginal Markov 
property with respect to $\G[D_i]$, $i=1,\ldots,k$.  
\end{prop}

The proof of this is omitted but see \citet{shpitser:14}, which
includes various examples. $q_i$ is a conditional distribution, but
can be renormalized as a joint density over $D_i \cup \pa_{\G}(D_i)$.
The notion of conditional distributions in graphical models is dealt
with in \citet{shpitser:14} by having two types of vertex, separately
representing the random and conditioned variables; we have omitted
these details for the sake of brevity.

\subsection{Weaker Markov Properties}

The marginal model precisely answers our original question: what
collections of distributions can be induced as the margin of a DAG
model?  However, because the definition is rather indirect, it is
generally difficult to characterize the set $\M_m(\G)$, and we may be
unable to tell whether or not a particular distribution lies in it or
not.  This complexity is one of the motivations behind the ordinary
and nested Markov properties of \citet{richardson:03} and
\citet{shpitser:14} respectively.  Both properties follow from
treating the ancestrality in Proposition \ref{prop:subgraph}(b) and
the factorization in Proposition \ref{prop:nested} as axiomatic.  In
order to do so, we assume the existence of a joint density with
respect to a product measure on $\X_V$.

\begin{dfn}
Let $\G$ be an mDAG with vertices $V$, and $P$ a probability distribution
over $\X_V$ with density $p$.  Say that $P$ obeys the 
\emph{nested Markov property} with respect to $P$ if either $|V|=1$, or both:
\begin{enumerate}[1.]
\item for every ancestral set $A \subseteq V$, the margin of $P$ over $X_A$ obeys the nested Markov property for $\G_A$; and
\item if $\G$ has districts $D_1, \ldots, D_k$ then
  $p(x_V) = \prod_{i=1}^k q_i(x_{D_i} \,|\, x_{\pa(D_i)\setminus D_i})$, where each
  $q_i$ obeys the nested Markov property for $\G[D_i]$.
\end{enumerate}
\end{dfn}

We denote the resulting models by $\M_n(\G)$.  The nested model
`throws away' the inequality constraints of the marginal model, but
for discrete variables is known to give models of the same dimension
\citep{evans:complete}, and it has the advantage of a fairly explicit
characterization.  Various equivalent formulations to the one above
are given in \citet{shpitser:14}.

The ordinary model can be defined in the same way as the nested model, 
but replacing 2 with the weaker condition:
\begin{enumerate}[1'.]
\addtocounter{enumi}{1}
\item if $\G$ has districts $D_1, \ldots, D_k$ then
  $p(x_V) = \prod_{i=1}^k q_i(x_{D_i} \,|\, x_{\pa(D_i)\setminus D_i})$ for some
  conditional densities $q_i$.
\end{enumerate}
Crucially, no further structure is imposed upon the pieces $q_i$, so
the definition does not recurse.  From their definitions and
Proposition \ref{prop:nested} it is clear that the models obey the
inclusion $\M_m(\G) \subseteq \M_n(\G) \subseteq \M_o(\G)$: the next
example show that these inclusions are strict in general.

\begin{exm}
  Consider again the graph in Figure \ref{fig:verma}; its latent
  projection over the vertices $\{1,2,3,4\}$ is shown in Figure
  \ref{fig:verma:mdag}(a): call this projection $\G$.  Applying the
  ancestrality property we see that, under the ordinary Markov
  property the margin over $(X_1, X_2, X_3)$ satisfies the global
  Markov property for the DAG $1 \rightarrow 2 \rightarrow 3$, so
  $X_1 \indep X_3 \,|\, X_2$.

  If we factorize into districts we find 
  \bal
  p(x_1, x_2, x_3, x_4) = q_{1}(x_1) \cdot q_{3}(x_3 \,|\, x_1, x_2) \cdot q_{24}(x_2, x_4 \,|\, x_1, x_3),
  \eal
  which is a vacuous requirement under the ordinary Markov property,
  and indeed there are no further constraints.  However, the nested
  property additionally requires that $q_{24}$ obeys the nested
  property for the mDAG in Figure \ref{fig:verma:mdag}(b).  Under this
  graph we see that $X_4 \indep X_1 \,|\, X_3$, and this gives the
  constraint (\ref{eqn:vermacons}); hence
  $\M_n(\G) \subset \M_o(\G)$.
  
\begin{figure}
 \begin{center}
 \begin{tikzpicture}
 [rv/.style={circle, draw, very thick, minimum size=6.5mm, inner sep=0.75mm}, node distance=20mm, >=stealth]
 \pgfsetarrows{latex-latex};
\begin{scope}
 \node[rv]  (1)              {$1$};
 \node[rv, right of=1] (2) {$2$};
 \node[rv, right of=2] (3) {$3$};
 \node[rv, right of=3] (4) {$4$};
 \draw[->, very thick, color=blue] (1) -- (2);
 \draw[->, very thick, color=blue] (2) -- (3);
 \draw[->, very thick, color=blue] (3) -- (4);
 \draw[<->, very thick, color=red] (2.45) .. controls +(1,1) and +(-1,1) .. (4.135);
\node[left of=1] {(a)};
\end{scope}
\begin{scope}[yshift=-3cm]
 \node[rv]  (1)              {$1$};
 \node[rv, right of=1] (2) {$2$};
 \node[rv, right of=2] (3) {$3$};
 \node[rv, right of=3] (4) {$4$};
 \draw[->, very thick, color=blue] (1) -- (2);
 \draw[->, very thick, color=blue] (3) -- (4);
 \draw[<->, very thick, color=red] (2.45) .. controls +(1,1) and +(-1,1) .. (4.135);
\node[left of=1] {(b)};
\end{scope}
 \end{tikzpicture}
 \caption{(a) An mDAG $\G$ representing the DAG in Figure \ref{fig:verma}, with
   the vertex 5 treated as unobserved.  (b) The subgraph $\G[\{2,4\}]$.}
 \label{fig:verma:mdag}
 \end{center}
\end{figure}
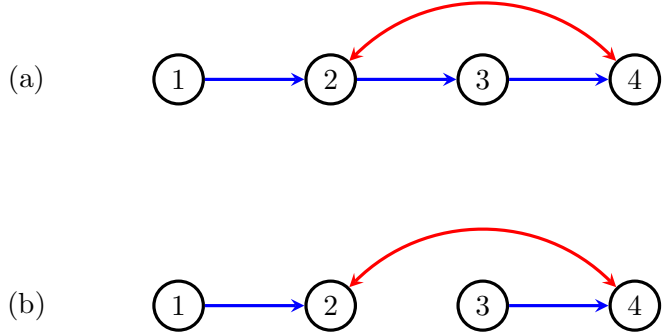

If $X_2$ and $X_4$ are discrete, then the marginal Markov property
induces an extra inequality constraint known as Bell's inequality
\citep{bell:64, gill:14}; hence $\M_m(\G) \subset \M_n(\G)$.
\end{exm}

\section{Markov Equivalent Graphs} \label{sec:markov}

A natural question to ask when two different graphs lead to the same
model under a particular Markov property.  That is, what is the
equivalence class determined by $\G \sim \G'$ whenever
$\M_m(\G) = \M_m(\G')$?  Without further assumptions such as a causal
ordering, graphs that are Markov equivalent are indistinguishable; any
model search procedure over the class of mDAG models should therefore
report the equivalence class rather than a single graph.  In addition,
because the marginal Markov property is difficult to characterize
explicitly, it can be helpful to reduce a problem down to a simpler
graph (see Example \ref{exm:equiv}).

For the ordinary Markov property there is a relatively simple
criterion for determining whether two graphs are equivalent
\citep{richardson:03}; for the nested Markov model, on the other hand,
equivalence is an open problem.  This section provides partial results
towards a characterization in the case of the marginal model.  We
conjecture that if two graphs are equivalent under the marginal
property then they are also equivalent under the nested property.  The
results of \citet{evans:complete} show that this holds for discrete
variables, but the general case is still open.

Our first substantive equivalence result generalizes an idea for
instrumental variables.

\begin{prop} \label{prop:equiv} 
Let $\G$ be an mDAG containing a bidirected facet $B = C \dot\cup D$
such that:
\begin{enumerate}[(i)]
\item every bidirected face containing any $c \in C$ is a subset of $B$;
  and
\item $\pa_\G(d) \supseteq \pa_\G(C)$ for each $d \in
  D$.
\end{enumerate} 
Let $\H$ be the mDAG defined from $\G$ by removing the facet $B$ and
replacing it with $C$ and $D$, and adding edges $c \rightarrow d$ for
each $c \in C$ and $d \in D$ (where such an edge is not already
present).

Then $\M_m(\G) = \M_m(\H)$.
\end{prop}

\begin{proof}
  The result follows from Lemma \ref{lem:splitnode} in the appendix,
  which shows that under these circumstances we can split the latent
  variable corresponding to $B$ into two independent pieces.
\end{proof}

\begin{exm}
  Consider the mDAG in Figure \ref{fig:reduce}(a).  We can apply the
  Proposition with $C = \{a,b\}$ and $D = \{c,d\}$ to see that it is
  Markov equivalent to the graph in Figure \ref{fig:reduce}(b).  The
  advantage of such a reduction is that it moves the graph `closer' to
  something which looks like a DAG, having smaller bidirected facets.
  This makes it clearer how the joint distribution factorizes.

 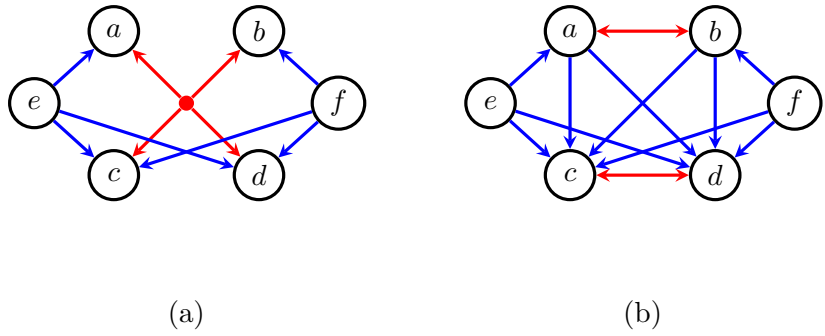
\begin{figure}
 \begin{center}
 \begin{tikzpicture}[rv/.style={circle, draw, very thick, minimum size=6.5mm, inner sep=1mm}, node distance=20mm, >=stealth]
 \pgfsetarrows{latex-latex};
\begin{scope}
 \node[circle, minimum size=2mm, inner sep=0mm, fill=red] (U2) {};
 \node[rv] (1) at (135:1.35) {$a$};
 \node[rv] (2) at (225:1.35) {$c$};
 \node[rv] (4) at (315:1.35) {$d$};
 \node[rv] (3) at (45:1.35) {$b$};
 \node[rv] (5) at (180:2) {$e$};
 \node[rv] (6) at (00:2) {$f$};
 \draw[->, very thick, color=red] (U2) -- (1);
 \draw[->, very thick, color=red] (U2) -- (2);
 \draw[->, very thick, color=red] (U2) -- (3);
 \draw[->, very thick, color=red] (U2) -- (4);
 \draw[->, very thick, color=blue] (6) -- (2);
 \draw[->, very thick, color=blue] (6) -- (3);
 \draw[->, very thick, color=blue] (6) -- (4);
 \draw[->, very thick, color=blue] (5) -- (1);
 \draw[->, very thick, color=blue] (5) -- (2);
 \draw[->, very thick, color=blue] (5) -- (4);
\node[below of=U2, yshift=-8mm] {(a)};
\end{scope}
\begin{scope}[xshift=6cm]
 \node[minimum size=0mm, inner sep=0mm] (U2) {};
 \node[rv] (1) at (135:1.35) {$a$};
 \node[rv] (2) at (225:1.35) {$c$};
 \node[rv] (4) at (315:1.35) {$d$};
 \node[rv] (3) at (45:1.35) {$b$};
 \node[rv] (5) at (180:2) {$e$};
 \node[rv] (6) at (0:2) {$f$};
 \draw[<->, very thick, color=red] (1) -- (3);
 \draw[<->, very thick, color=red] (2) -- (4);
 \draw[->, very thick, color=blue] (1) -- (2);
 \draw[->, very thick, color=blue] (1) -- (4);
 \draw[->, very thick, color=blue] (3) -- (2);
 \draw[->, very thick, color=blue] (3) -- (4);
 \draw[->, very thick, color=blue] (6) -- (2);
 \draw[->, very thick, color=blue] (6) -- (3);
 \draw[->, very thick, color=blue] (6) -- (4);
 \draw[->, very thick, color=blue] (5) -- (1);
 \draw[->, very thick, color=blue] (5) -- (2);
 \draw[->, very thick, color=blue] (5) -- (4);
\node[below of=U2, yshift=-8mm] {(b)};
\end{scope}
 \end{tikzpicture}
 \caption{Two mDAGs shown to be Markov equivalent by application of Proposition \ref{prop:equiv}} 
\label{fig:reduce}
  \end{center}
\end{figure}
\end{exm}

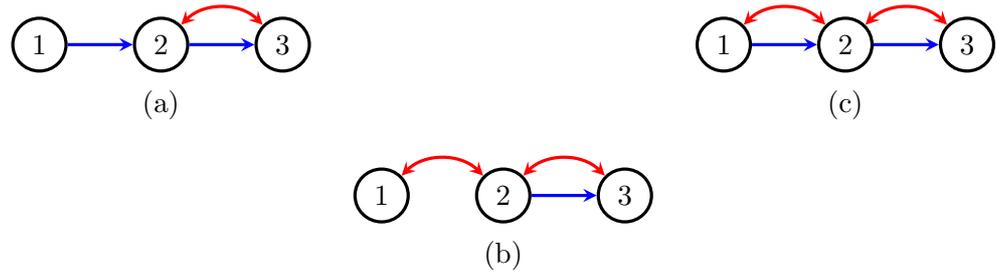
\begin{figure}
\begin{center}
\begin{tikzpicture}
 [rv/.style={circle, draw, very thick, minimum size=7mm, inner sep=0.75mm}, node distance=16mm, >=stealth]
 \pgfsetarrows{latex-latex};
\begin{scope}
 \node[minimum size=0mm, inner sep=0mm] (U) {};
 \node[rv] (1)  {$1$};
 \node[rv, right of=1] (2)  {$2$};
 \node[rv, right of=2] (3)  {$3$};
 \draw[->, very thick, color=blue] (1) -- (2);
 \draw[->, very thick, color=blue] (2) -- (3);
 \draw[<->, very thick, color=red] (2.45) .. controls +(0.3,0.3) and +(-0.3,0.3) .. (3.135);
 \node[below of=2, yshift=8mm] {(a)};
\end{scope}
\begin{scope}[xshift=4.5cm, yshift=-2cm]
 \node[minimum size=0mm, inner sep=0mm] (U) {};
 \node[rv] (1)  {$1$};
 \node[rv, right of=1] (2)  {$2$};
 \node[rv, right of=2] (3)  {$3$};
 \draw[->, very thick, color=blue] (2) -- (3);
 \draw[<->, very thick, color=red] (1.45) .. controls +(0.3,0.3) and +(-0.3,0.3) .. (2.135);
 \draw[<->, very thick, color=red] (2.45) .. controls +(0.3,0.3) and +(-0.3,0.3) .. (3.135);
 \node[below of=2, yshift=8mm] {(b)};
\end{scope}
\begin{scope}[xshift=9cm]
 \node[minimum size=0mm, inner sep=0mm] (U) {};
 \node[rv] (1)  {$1$};
 \node[rv, right of=1] (2)  {$2$};
 \node[rv, right of=2] (3)  {$3$};
 \draw[->, very thick, color=blue] (1) -- (2);
 \draw[->, very thick, color=blue] (2) -- (3);
 \draw[<->, very thick, color=red] (1.45) .. controls +(0.3,0.3) and +(-0.3,0.3) .. (2.135);
 \draw[<->, very thick, color=red] (2.45) .. controls +(0.3,0.3) and +(-0.3,0.3) .. (3.135);
 \node[below of=2, yshift=8mm] {(c)};
\end{scope}
 \end{tikzpicture}
\caption{Three Markov equivalent graphs representing the instrumental variables model.}
\label{fig:inst2}
\end{center}
\end{figure}

\begin{exm} \label{exm:iv} The canonical example to which Proposition
  \ref{prop:equiv} can be applied is the \emph{instrumental variables}
  model, shown in Figure \ref{fig:inst2}(a).  As noted by
  \citet{didelez:07}, it is not possible observationally to tell
  whether 1 is a direct cause of 2, or there is a hidden common cause,
  or both.  Applying Proposition \ref{prop:equiv} to the graphs in
  Figure \ref{fig:inst2}(b) and (c) with $C = \{1\}$ and $D = \{2\}$
  shows that they are indeed equivalent to Figure \ref{fig:inst2}(a).
\end{exm}

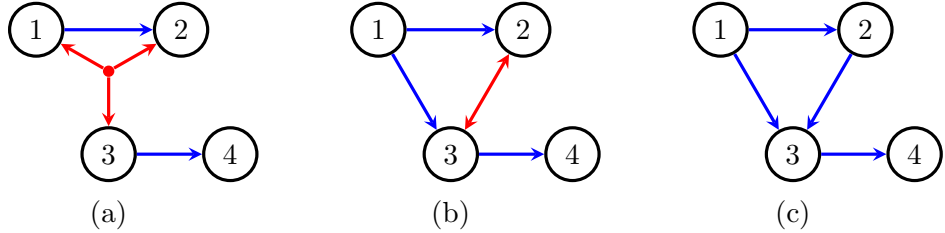
\begin{figure}
\begin{center}
\begin{tikzpicture}
 [rv/.style={circle, draw, very thick, minimum size=7mm, inner sep=0.75mm}, node distance=16mm, >=stealth]
 \pgfsetarrows{latex-latex};
\begin{scope}
 \node[circle, minimum size=1.5mm, inner sep=0mm, fill=red] (U) {};
 \node[rv] (1) at (150:1.1) {$1$};
 \node[rv] (2) at (270:1.1) {$3$};
 \node[rv] (3) at (30:1.1) {$2$};
 \node[rv, right of=2] (4) {$4$};
 \draw[->, very thick, color=red] (U) -- (1);
 \draw[->, very thick, color=red] (U) -- (2);
 \draw[->, very thick, color=red] (U) -- (3);
 \draw[->, very thick, color=blue] (1) -- (3);
 \draw[->, very thick, color=blue] (2) -- (4);
 \node[below of=2, yshift=8mm] {(a)};
\end{scope}
\begin{scope}[xshift=4.5cm]
 \node[minimum size=0mm, inner sep=0mm] (U) {};
 \node[rv] (1) at (150:1.1) {$1$};
 \node[rv] (2) at (270:1.1) {$3$};
 \node[rv] (3) at (30:1.1) {$2$};
 \node[rv, right of=2] (4) {$4$};
 \draw[->, very thick, color=blue] (1) -- (3);
 \draw[->, very thick, color=blue] (1) -- (2);
 \draw[<->, very thick, color=red] (2) -- (3);
 \draw[->, very thick, color=blue] (2) -- (4);
 \node[below of=2, yshift=8mm] {(b)};
\end{scope}
\begin{scope}[xshift=9cm]
 \node[minimum size=0mm, inner sep=0mm] (U) {};
 \node[rv] (1) at (150:1.1) {$1$};
 \node[rv] (2) at (270:1.1) {$3$};
 \node[rv] (3) at (30:1.1) {$2$};
 \node[rv, right of=2] (4) {$4$};
 \draw[->, very thick, color=blue] (1) -- (3);
 \draw[->, very thick, color=blue] (1) -- (2);
 \draw[->, very thick, color=blue] (3) -- (2);
 \draw[->, very thick, color=blue] (2) -- (4);
 \node[below of=2, yshift=8mm] {(c)};
\end{scope}
 \end{tikzpicture}
\caption{(a) An mDAG; (b) an mDAG which is Markov equivalent to the one in (a); and (c) a DAG which is Markov equivalent to the mDAGs.}
\label{fig:inst1}
\end{center}
\end{figure}

\begin{exm} \label{exm:equiv}
  The mDAG in Figure \ref{fig:inst1}(a) can be reduced to the simpler
  one in \ref{fig:inst1}(b) by applying Proposition \ref{prop:equiv}
  with $C = \{1\}$ and $D = \{2,3\}$.  This can be further simplified
  to the DAG in (c) by applying the proposition again, this time with
  $C = \{2\}$ and $D = \{3\}$.  By using the global Markov property for
  DAGs, this shows that each graph represents those distributions
  under which $X_4 \indep X_1, X_2 \,|\, X_3$.
\end{exm}

Define the \emph{skeleton} of an mDAG $\G(V,\mathcal{E}, \mathcal{B})$
as the simple undirected graph with vertices $V$, and edges $v - w$
whenever $v$ and $w$ appear together in some edge (directed or
bidirected) in $\G$.

\begin{prop} \label{prop:skel}
Let $\G$ and $\H$ be mDAGs with different skeletons.  Then 
if the state-space $\X_V$ is discrete 
$\M_m(\G) \neq \M_m(\H)$.
\end{prop}

\begin{proof}
This follows from \citet{evans:12}, Corollary 4.4.
\end{proof}

Note that this is not necessarily true for all state-spaces: if $X_2$
is continuous the three models defined by applying the marginal 
Markov property to the graphs in Figure \ref{fig:inst2} are all
saturated (i.e.\ contain any joint distribution over those variables),
even though they have skeleton $1-2-3$ \citep{bonet:01}.

\subsection{Bidirected Graphs and Connection to
  ADMGs} \label{sec:admgs}

The notion of latent projection was defined by \citet{verma:91a} with
respect to acyclic directed mixed graphs (though this term for such graphs 
was not introduced until \citet{richardson:03}).  The importance of our more
general formulation is now made clear.  

\begin{exm} \label{exm:3cycle} Consider the mDAGs in Figure
  \ref{fig:3cycle}.  The graph in Figure \ref{fig:3cycle}(a) is the
  latent projection one would obtain from a single latent variable
  having all three observed nodes as children, while Figure
  \ref{fig:3cycle}(b) corresponds to having three independent latents,
  each with a pair of observed variables as children.  The first graph
  is associated with a model which is clearly saturated, but the
  second is not: for example, if the observed variables are binary, it
  is not possible to have $P(X_1 = X_2 = X_3 = 1) = P(X_1 = X_2 = X_3
  = 0) = \frac{1}{2}$ \citep{fritz:12}.

\begin{figure}
\begin{center}
\begin{tikzpicture}
 [rv/.style={circle, draw, very thick, minimum size=7mm, inner sep=0.75mm}, node distance=20mm, >=stealth]
 \pgfsetarrows{latex-latex};
\begin{scope}
 \node[circle, minimum size=2mm, inner sep=0mm, fill=red] (U) {};
 \node[rv] (1) at (150:1.3) {$1$};
 \node[rv] (2) at (270:1.3) {$2$};
 \node[rv] (3) at (30:1.3) {$3$};
 \draw[->, very thick, color=red] (U) -- (1);
 \draw[->, very thick, color=red] (U) -- (2);
 \draw[->, very thick, color=red] (U) -- (3);
 \node[below of=2, yshift=10mm] {(a)};
\end{scope}
\begin{scope}[xshift=5cm]
 \node[minimum size=0mm, inner sep=0mm] (U) {};
 \node[rv] (1) at (150:1.1) {$1$};
 \node[rv] (2) at (270:1.1) {$2$};
 \node[rv] (3) at (30:1.1) {$3$};
 \draw[<->, very thick, color=red] (3.140) .. controls +(140:0.5) and +(40:0.5) .. (1.40);
 \draw[<->, very thick, color=red] (1.260) .. controls +(260:0.5) and +(160:0.5) .. (2.160);
 \draw[<->, very thick, color=red] (2.20) .. controls +(20:0.5) and +(280:0.5) .. (3.280);
 \node[below of=2, yshift=8mm] {(b)};
\end{scope}
 \end{tikzpicture}
\caption{(a) An mDAG corresponding to a saturated model; (b) an mDAG corresponding to a model with constraints.}
\label{fig:3cycle}
\end{center}
\end{figure}
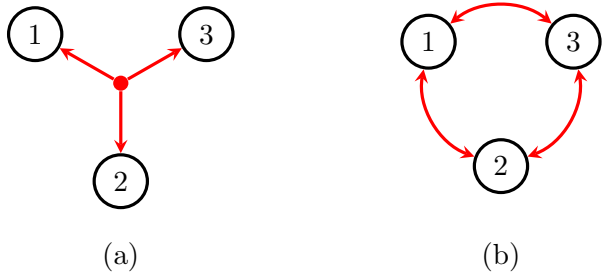
\end{exm}

Under Verma's original formulation of latent projection with ADMGs,
both these models are represented by the same graph: the one in Figure
\ref{fig:3cycle}(b).  However, as the previous example shows, the two
marginal models formed in this way are actually distinct.  The next
result generalizes this idea.

\begin{lem} \label{lem:notsat} Let $\G$ be a purely bidirected mDAG
  with vertices $V$, whose bidirected faces consist of all non-empty
  $B \subset V$ strict subsets of vertices.  Then the model $\M(\G)$
  is not saturated (for any state-space $\X_V$).
\end{lem}

\begin{proof} 
For each $v \in V$, let $B_v = V \setminus \{v\}$, so that
$\mathcal{B}$ consists of the sets $B_v$ and their subsets.  The
canonical DAG for $\bar\G$ has vertices $V \cup \{B_v : v \in V\}$ and
edges $B_v \rightarrow w$ whenever $v \neq w$.

Let $(X_V, Y_{\mathcal{B}})$ have a joint distribution which respects
the SEP with respect to $\bar\G$, so that, writing $\bs Y_{-v} \equiv
(Y_{B_w} : w \neq v)$, we have $X_v = f_v(\bs Y_{-v},E_v)$.

Given some permutation $s$ of $V$ such that $s(v) \neq v$ for any $v
\in V$, let $\mathcal{F}_v = \sigma(Y_{B_v}, E_{s(v)})$.  Note that each $X_v$ is
$\sigma\left(\bigvee_{w \neq v} \mathcal{F}_w\right)$-measurable, and that
all the $\sigma$-algebrae $\mathcal{F}_v$ are independent.

It follows from Lemma \ref{lem:meas} in the appendix that if $\E(X_v -
X_w)^2 \leq \epsilon$ for each $v,w$, then each $X_v$ has variance at
most $|V| \epsilon$.  But this precludes, for example, the possibility
of a joint binary distribution in which $P(\{X_v \text{ all equal}\})
= 1-\epsilon$ with $P(X_v = 0) = P(X_v = 1) = \frac{1}{2}$ for some
sufficiently small positive $\epsilon$.  Since it is always possible
to dichotomize a (non-trivial) random variable, this shows that the
model is not saturated on any state-space.
%
%
\end{proof}

In the case where mDAGs contain only bidirected edges, Markov
equivalence turns out to be very simple.

\begin{prop} \label{prop:bidi} Let $\G, \G'$ be mDAGs containing no
  directed edges.  Then $\M_m(\G) = \M_m(\G')$ if and only if
  $\G = \G'$.
\end{prop}

\begin{proof}
  Suppose that $\G$ and $\G'$ are not equal, so (without loss of
  generality) there exists some $B \in \mathcal{B}(\G) \setminus
  \mathcal{B}(\G')$.  Since $B$ is ancestral (there are no directed
  edges), it is sufficient to prove that $\M_m(\G_B) \neq
  \M_m(\G'_B)$, so assume that in fact the vertices of $\G$ and $\G'$
  are $B$.  The model $\M_m(\G)$ is saturated.

  Let $\tilde\G$ be the bidirected graph with vertices $B$ and such
  that $\mathcal{B}(\tilde\G)$ consists of all strict subsets of $B$;
  by Lemma \ref{lem:notsat} $\M_m(\tilde\G)$ is not saturated.  But
  $\G' \subseteq \tilde\G$, so $\M_m(\G') \subseteq \M_m(\tilde\G)
  \subset \M_m(\G)$, so in particular $\M_m(\G) \neq \M_m(\G')$
\end{proof}

It follows from this result that ordinary graphs are fundamentally
unable to fully represent marginal models, even if we add additional
kinds of edge; the number of possible marginal models just grows too
quickly.  Consequently our extension to hyper-edges is necessary.

\begin{cor}
No class of ordinary graphs (i.e.\ not hyper-graphs) is sufficient 
to represent marginal models of DAGs.
\end{cor}

\begin{proof}
  The number of simplicial complexes on $n$ vertices grows faster than
  $2^{n \choose \lfloor n/2 \rfloor}$ \citep[see, for
  example,][]{kleitman:69}, so by Proposition \ref{prop:bidi} there
  are at least this many marginal models.  For a class of ordinary
  graphs with $k$ different edge types, there are only
  $2^{k{n \choose 2}}$ different graphs, and
  ${n \choose \lfloor n/2 \rfloor} > k{n \choose 2}$ for sufficiently
  large $n$.  Hence ordinary graphs are not sufficient.
\end{proof}

\subsection{mDAGs on Three Variables}

There are 48 distinct mDAGs over three unlabelled vertices (i.e.\ up
to permutation of the vertices).  Using Propositions
\ref{prop:subgraph}, \ref{prop:equiv} and \ref{prop:skel} one can show
that of these there are 8 equivalence classes of induced models.
These are shown in Figure \ref{fig:three}.  Five of them are DAG
models, the other three being the instrumental variables model from
Figure \ref{fig:inst2}(a), the `unrelated confounding' model studied
by \citet{evans:12}, and the pairwise bidirected model from Example
\ref{exm:3cycle}.

\newlength{\nodesep}
\setlength{\nodesep}{13mm}
\newlength{\blocksep}
\setlength{\blocksep}{35mm}

\newcommand{\xvar}{}
\newcommand{\yvar}{}
\newcommand{\zvar}{}

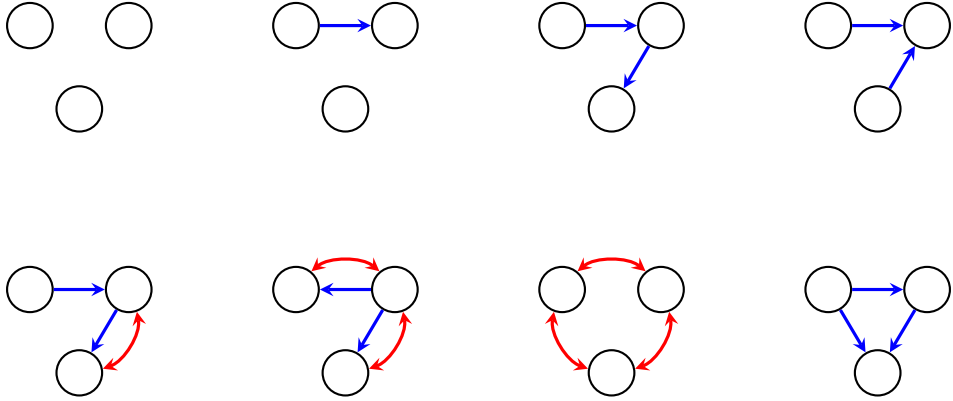
\begin{figure}
\begin{center}
\begin{tikzpicture}
[rv/.style={draw, circle, thick, minimum size=6mm, inner sep=0.8mm}, node distance=\nodesep, >=stealth]
 \pgfsetarrows{latex-latex};
\begin{scope}
\node[rv]  (x1)            {\xvar};
\node[rv, right of=x1] (x2) {\yvar};
\node[rv, below of=x2, xshift=-\nodesep/2, yshift=\nodesep*0.15] (x3) {\zvar};
\end{scope}
\begin{scope}[xshift=\blocksep]
\node[rv]  (x1)            {\xvar};
\node[rv, right of=x1] (x2) {\yvar};
\node[rv, below of=x2, xshift=-\nodesep/2, yshift=\nodesep*0.15] (x3) {\zvar};
\draw[->, very thick, color=blue] (x1) -- (x2);
\end{scope}
\begin{scope}[xshift=\blocksep*2]
\node[rv]  (x1)            {\xvar};
\node[rv, right of=x1] (x2) {\yvar};
\node[rv, below of=x2, xshift=-\nodesep/2, yshift=\nodesep*0.15] (x3) {\zvar};
\draw[->, very thick, color=blue] (x1) -- (x2);
\draw[->, very thick, color=blue] (x2) -- (x3);
\end{scope}
\begin{scope}[xshift=\blocksep*3]
\node[rv]  (x1)            {\xvar};
\node[rv, right of=x1] (x2) {\yvar};
\node[rv, below of=x2, xshift=-\nodesep/2, yshift=\nodesep*0.15] (x3) {\zvar};
\draw[->, very thick, color=blue] (x1) -- (x2);
\draw[<-, very thick, color=blue] (x2) -- (x3);
\end{scope}
\begin{scope}[yshift=-\blocksep]
\node[rv]  (x1)            {\xvar};
\node[rv, right of=x1] (x2) {\yvar};
\node[rv, below of=x2, xshift=-\nodesep/2, yshift=\nodesep*0.15] (x3) {\zvar};
\draw[->, very thick, color=blue] (x1) -- (x2);
\draw[->, very thick, color=blue] (x2) -- (x3);
\draw[<->, very thick, color=red] (x2.290) .. controls +(-80:0.3) and +(20:0.3) .. (x3.10);
\end{scope}
\begin{scope}[xshift=\blocksep, yshift=-\blocksep]
\node[rv]  (x1)            {\xvar};
\node[rv, right of=x1] (x2) {\yvar};
\node[rv, below of=x2, xshift=-\nodesep/2, yshift=\nodesep*0.15] (x3) {\zvar};
\draw[<-, very thick, color=blue] (x1) -- (x2);
\draw[->, very thick, color=blue] (x2) -- (x3);
\draw[<->, very thick, color=red] (x1.50) .. controls +(40:0.3) and +(140:0.3) .. (x2.130);
\draw[<->, very thick, color=red] (x2.290) .. controls +(-80:0.3) and +(20:0.3) .. (x3.10);
\end{scope}
\begin{scope}[xshift=\blocksep*2, yshift=-\blocksep]
\node[rv]  (x1)            {\xvar};
\node[rv, right of=x1] (x2) {\yvar};
\node[rv, below of=x2, xshift=-\nodesep/2, yshift=\nodesep*0.15] (x3) {\zvar};
\draw[<->, very thick, color=red] (x1.50) .. controls +(40:0.3) and +(140:0.3) .. (x2.130);
\draw[<->, very thick, color=red] (x2.290) .. controls +(-80:0.3) and +(20:0.3) .. (x3.10);
\draw[<->, very thick, color=red] (x3.170) .. controls +(160:0.3) and  +(-100:0.3) .. (x1.250);
\end{scope}
\begin{scope}[xshift=\blocksep*3, yshift=-\blocksep]
\node[rv]  (x1)            {\xvar};
\node[rv, right of=x1] (x2) {\yvar};
\node[rv, below of=x2, xshift=-\nodesep/2, yshift=\nodesep*0.15] (x3) {\zvar};
\draw[->, very thick, color=blue] (x1) -- (x2);
\draw[->, very thick, color=blue] (x1) -- (x3);
\draw[->, very thick, color=blue] (x2) -- (x3);
\end{scope}
\end{tikzpicture}
\caption{mDAGs representing the eight distinct models over three
  (unlabelled) variables.}
\label{fig:three}
\end{center}
\end{figure}

For four nodes the problem becomes much more complicated.  As an
illustration of the limitations of the results in this section, we
note that we are unable to determine whether or not the graphs in
Figure \ref{fig:unknown} represent saturated models under the marginal
Markov property or not.

\begin{figure}
\begin{center}
\begin{tikzpicture}
 [rv/.style={circle, draw, very thick, minimum size=7mm, inner sep=0.75mm}, node distance=16mm, >=stealth]
 \pgfsetarrows{latex-latex};
\begin{scope}[yshift=1cm]
 \node[rv] (1)  {$1$};
 \node[rv, right of=1] (2) {$2$};
 \node[rv, below of=1] (3) {$3$};
 \node[rv, right of=3] (4) {$4$};
 \draw[<->, very thick, color=red] (1) -- (2);
 \draw[<->, very thick, color=red] (2) -- (4);
 \draw[<->, very thick, color=red] (1) -- (3);
 \draw[<->, very thick, color=red] (4) -- (3);
 \draw[->, very thick, color=blue] (1) -- (4);
 \draw[->, very thick, color=blue] (2) -- (3);
 \node[below of=3, xshift=8mm, yshift=8mm] {(a)};
\end{scope}
\begin{scope}[xshift=4cm, yshift=-3mm]
 \node[rv] (1) {$1$};
 \node[rv, right of=1] (2) {$2$};
 \node[rv, right of=2] (4) {$4$};
 \node[rv, above of=2, yshift=-2mm] (3) {$3$};
 \draw[->, very thick, color=blue] (1) -- (2);
 \draw[->, very thick, color=blue] (2) -- (4);
 \draw[->, very thick, color=blue] (3) -- (4);
 \draw[<->, very thick, color=red] (1) -- (3);
 \draw[<->, very thick, color=red] (2) -- (3);
 \draw[<->, very thick, color=red] (1) .. controls +(1,-.7) and +(-1,-.7) .. (4);
 \node[below of=2, yshift=5mm] {(b)};
\end{scope}
\begin{scope}[xshift=9cm, yshift=2mm]
 \node[rv] (1) {$1$};
 \node[rv, right of=1] (2) {$2$};
 \node[rv, right of=2] (3) {$3$};
 \node[rv, right of=3] (4) {$4$};
 \draw[->, very thick, color=blue] (1) -- (2);
 \draw[->, very thick, color=blue] (2) -- (3);
 \draw[->, very thick, color=blue] (3) -- (4);
 \draw[<->, very thick, color=red] (1.330) .. controls +(.7,-.7) and +(-.7,-.7) .. (3.210);
 \draw[<->, very thick, color=red] (2.30) .. controls +(.7,.7) and +(-.7,.7) .. (4.150);
 \draw[<->, very thick, color=red] (1.320) .. controls +(1,-1) and +(-1,-1) .. (4.220);
 \node[below of=2, xshift=8mm, yshift=0mm] {(c)};
\end{scope}
 \end{tikzpicture}
\caption{Three mDAGs whose associated models under the marginal Markov property may or may not be saturated.}
\label{fig:unknown}
\end{center}
\end{figure}
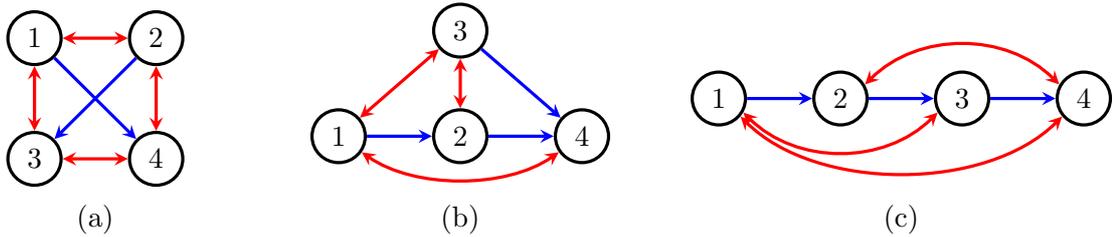

\section{Causal Models and Interventions} \label{sec:causal}

The use of DAGs to represent causal models goes back to the work of
Sewall Wright, and has found popularity more recently \citep[see][and
references therein]{sgs:00, pearl:09}.  The use of an arrow $X
\rightarrow Y$ to express the statement that `$X$ causes $Y$' is
natural and intuitive, and directed acyclic graphs provide a
convenient recursive structure for representing causal models, with
acyclicity enforcing the idea that causes should precede effects in
time.

Note that the structural equation property as formulated in Definition
\ref{dfn:sep} only posits the existence of \emph{some} functions $f_v$
and error variables $E_v$ which generate the required joint
distribution.  In general, there will be many graphical structures and
pairs $(f_v, E_v)$ which give rise to a given distribution.  However,
if a distribution is \emph{structurally} generated in this way, then
when some of the variables in the system are intervened upon (in an
appropriately defined way), a suitably modified version of the
original DAG will correctly represent the resulting interventional
probability distribution \citep{pearl:09}.  Analogously we will show
that mDAGs are able to represent the models induced on the
\emph{margins} of DAGs after intervention.


\begin{dfn}
  Let $\D$ be a DAG with vertices $V$, and suppose that data are
  generated according to a particular collection of pairs $(f_v,
  E_v)$, $v \in V$ which satisfy the SEP for $\D$.
  An \emph{intervention} on $A \subseteq V$ replaces $(f_v, E_v)$ with
  $(\tilde{f}_v, \tilde{E}_v)$ for each $v \in A$, where $\tilde{f}_v
  : \mathscr{E}_v \rightarrow \X_v$ is measurable, and all $E_w$,
  $\tilde{E}_v$ are independent.

  Denote by $\D_{\overline{A}}$ the \emph{DAG $\D$ after intervening
    on $A$}, formed from $\D$ by removing edges directed towards $v
  \in A$.
\end{dfn}

An intervention removes the dependence of a variable on all of its
parents.  If $P$ is generated by $(f_v, E_v)$ according to the DAG
$\D$, then the distribution $P_{\overline{A}}$ after intervention on
$A$ is generated according to the \emph{mutilated} DAG
$\D_{\overline{A}}$, and hence obeys the SEP for
$\M(\D_{\overline{A}})$.  This definition of an intervention is based
on the one in \citet{pearl:09}.

Note that intervention is not a purely probabilistic operation, in the
sense that its effect it is not identifiable from the observed
probability distribution alone: it relies upon knowledge of the full
structural generating system.

\subsection{Causal mDAGs}

Let $\D$ be a DAG with vertex set $U \dot\cup V$ and let $\G =
\mathfrak{p}(\D, V)$.  If $(X_U,X_V)$ are generated according to the
structural equation property for $\D$, the definitions and results of
previous sections tell us that the distribution of $X_V$, say $P$, is
contained in $\M_m(\G)$.  If an intervention is performed on some of
the vertices in $V$, what then should we expect from the resulting
marginal distribution?

\begin{dfn}
  Let $\G(V,\mathcal{E},\mathcal{B})$ be an mDAG, and $A \subseteq V$.
  The mDAG $\G_{\overline{A}}$ has vertices $V$, directed edges
  $\mathcal{E}_{\overline{A}} = \{(w,v) \in \mathcal{E} : v \notin
  A\}$,
  and bidirected faces $\{B \setminus A : B \in \mathcal{B}\}$
  (together with the singletons $\{a\}$ for $a \in A$).
\end{dfn}

 \begin{figure}
 \begin{center}
 \begin{tikzpicture}[rv/.style={circle, draw, very thick, minimum size=6.5mm, inner sep=1mm}, node distance=20mm, >=stealth]
 \pgfsetarrows{latex-latex};
\begin{scope}
 \node (U2) {};
 \node[rv] (2) at (150:1.35) {$c$};
 \node[rv] (4) at (270:1.35) {$e$};
 \node[rv] (3) at (30:1.35) {$d$};
 \node[yshift=-13.5mm] (U3) at (3) {};
 \node[rv] (6) at ($(U3)+(330:1.35)$) {$f$};
 \node[rv, left of=2] (1) {$a$};
 \node[rv, left of=4, xshift=5mm] (5) {$b$};
 \draw[<->, very thick, color=red] (1) -- (2);
 \draw[->, very thick, color=blue] (2) -- (4);
 \draw[->, very thick, color=blue] (3) -- (6);
 \draw[->, very thick, color=blue] (5) -- (4);
 \draw[->, very thick, color=blue] (3) -- (4);
 \draw[<->, very thick, color=red] (2) .. controls +(330:1) and +(90:1) .. (4);
 \draw[<->, very thick, color=red] (6) .. controls +(150:1) and +(30:1) .. (4);
 \end{scope}
\end{tikzpicture}
\caption{The mDAG from Figure \ref{fig:mdag} after intervening on $d$.}
\label{fig:mdaginter}
\end{center}
\end{figure}
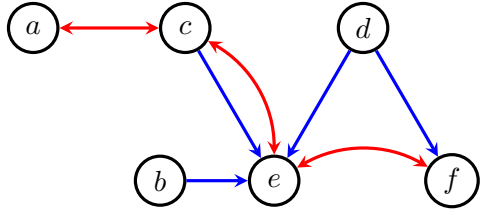

In other words to obtain $\G_{\overline{A}}$ from $\G$, delete
directed edges pointing to $A$, and remove vertices in $A$ from each
bidirected edge.  For example Figure \ref{fig:mdaginter} shows the
result of intervening on $\{d\}$ in the mDAG from Figure
\ref{fig:mdag}.  The next result shows that this definition of a
mutilated mDAG is sensible, because mutilation and projection 
commute.

\begin{prop} \label{prop:interven}
  Let $A \subseteq V$.  If $\G = \mathfrak{p}(\D, V)$, then
  $\G_{\overline{A}} = \mathfrak{p}(\D_{\overline{A}}, V)$.
\end{prop}

\begin{proof}
  Note that the definition of latent projections and of hidden common
  causes refer only to directed paths with non-endpoint vertices in
  $U$.  Since $U \cap A = \emptyset$, it follows that such a directed
  path in $\D$ is also contained in $\D_{\overline{A}}$ if and only if
  the final vertex is not in $A$.  Hence, the directed edges in
  $\mathfrak{p}(\D_{\overline{A}}, V)$ are precisely those which are
  in $\G = \mathfrak{p}(\D, V)$ and do not point to $A$, as required.

  Now, suppose $B \in \mathcal{B}(\G_{\overline{A}})$: then there is
  some $B' \in \mathcal{B}(\G)$ with $B' \setminus A = B$.  Hence $B'$
  share a hidden common cause in $\D$ with respect to $U$, and by the
  same reasoning as above, the vertices in $B' \setminus A = B$ share
  a hidden common cause in $\D_{\overline{A}}$ with respect to $U$.
  Hence $B \in \mathcal{B}(\mathfrak{p}(\D_{\overline{A}}, V))$

  Conversely, if $B \in \mathcal{B}(\mathfrak{p}(\D_{\overline{A}},
  V))$, then the elements of $B$ share a hidden common cause in
  $\D_{\overline{A}}$ with respect to $U$, and hence also in the
  supergraph $\D$.  So there is some $B' \supseteq B$ with $B'
  \setminus A = B$ such that $B' \in \mathcal{B}(\G)$, and hence $B
  \in \mathcal{B}(\G_{\overline{A}})$.
\end{proof}

It follows from this result that mDAGs not only represent the
structure of a margin of a DAG model, but they can also correctly
represent the manner in which it will change under interventions on
the observed variables.  

\begin{prop}
  Let $\D, \D'$ be DAGs with the same latent
  projection $\G$ over some set of variables $V$.  For any subset $A
  \subseteq V$ of intervened nodes, $\M(\D_{\overline{A}},
  V) = \M(\D'_{\overline{A}}, V)$
\end{prop}

\begin{proof}
  By Proposition \ref{prop:interven}, $\mathfrak{p}(\D_{\overline{A}},
  V) = \mathfrak{p}(\D'_{\overline{A}}, V)$, so that the result
  follows from Theorem \ref{thm:equiv}.
\end{proof}

Two DAGs may be observationally Markov equivalent, such as the graphs
$1 \rightarrow 2$ and $1 \leftarrow 2$ (which both represent saturated
models).  However, for any two distinct causal DAGs, there is always
some intervention under which the resulting mutilated DAGs are not
Markov equivalent.  For example, if we intervene on $1$ in the causal
model $1 \leftarrow 2$ the two variables become independent, but in
$1 \rightarrow 2$ the model remains unchanged.

We might hope that something similar holds for mDAGs: given distinct
mDAGs $\G, \H$, is there always some intervention such that
$\M_m(\G_{\overline{A}}) \neq \M_m(\H_{\overline{A}})$, so that one
could in principle distinguish between the two causal models via a
suitable experiment?  In fact this turns out not to be the case:
consider the mDAGs in Figures \ref{fig:3cycle2}(a) and (b); denote
then by $\G$ and $\H$ respectively.  Both represent saturated models,
so in particular $\M_m(\G) = \M_m(\H)$.  In addition, after
intervening on any of the vertices the resulting mutilated graphs are
the same: $\G_{\overline{A}} = \H_{\overline{A}}$ for any
$A \neq \emptyset$.  Hence
$\M(\G_{\overline{A}}) = \M(\H_{\overline{A}})$ for any
$A \subseteq \{1,2,3\}$.

The next result shows that two causal mDAGs can be distinguished by
intervention if they have different underlying DAGs.

\begin{prop}
  Let $\G$ and $\H$ be mDAGs on the same vertex set $V$, and suppose
  that their underlying DAGs are distinct.  Then for some $A \subseteq
  V$, $\M_m(\G_{\overline{A}}) \neq \M_m(\H_{\overline{A}})$.
\end{prop}

\begin{proof}
  Suppose that the edge $v \rightarrow w$ appears in $\G$ but not
  $\H$.  Then let $A = V \setminus \{w\}$: since non-trivial
  bidirected faces contain at least two vertices, $\G_{\overline{A}}$
  and $\H_{\overline{A}}$ are DAGs.  Therefore the only edges in
  $\G_{\overline{A}}$ and $\H_{\overline{A}}$ are those directed into
  $w$.  It follows that $X_v \indep X_w$ under any distribution in
  $\M_m(\H_{\overline{A}})$, whereas any form of dependence between
  $X_v$ and $X_w$ is possible in $\M_m(\G_{\overline{A}})$.
\end{proof}

\begin{rmk}
  The inability to distinguish between certain causal mDAGs is partly
  an artefact of the sort of interventions we consider.  If we allow
  more delicate interventions which can block a specific causal
  mechanism between any pair of variables, this would correspond to
  removing individual directed edges from the graph.  In this case, by
  blocking all the direct causal links we would obtain a distribution
  which satisfies the marginal Markov property for the underlying
  bidirected graphs.  It would then follow from Proposition
  \ref{prop:bidi} that causal models would be in one-to-one
  correspondence with graphs.

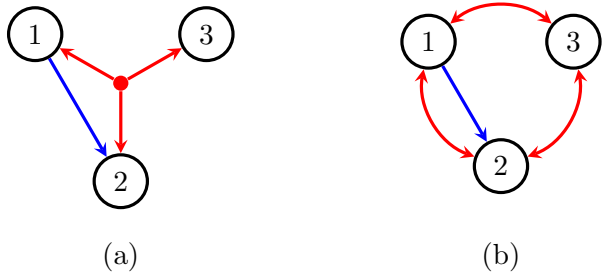
\begin{figure}
\begin{center}
\begin{tikzpicture}
 [rv/.style={circle, draw, very thick, minimum size=7mm, inner sep=0.75mm}, node distance=20mm, >=stealth]
 \pgfsetarrows{latex-latex};
\begin{scope}
 \node[circle, minimum size=2mm, inner sep=0mm, fill=red] (U) {};
 \node[rv] (1) at (150:1.3) {$1$};
 \node[rv] (2) at (270:1.3) {$2$};
 \node[rv] (3) at (30:1.3) {$3$};
 \draw[->, very thick, color=red] (U) -- (1);
 \draw[->, very thick, color=red] (U) -- (2);
 \draw[->, very thick, color=red] (U) -- (3);
 \draw[->, very thick, color=blue] (1) -- (2);
 \node[below of=2, yshift=10mm] {(a)};
\end{scope}
\begin{scope}[xshift=5cm]
 \node[minimum size=0mm, inner sep=0mm] (U) {};
 \node[rv] (1) at (150:1.1) {$1$};
 \node[rv] (2) at (270:1.1) {$2$};
 \node[rv] (3) at (30:1.1) {$3$};
 \draw[<->, very thick, color=red] (3.140) .. controls +(140:0.5) and +(40:0.5) .. (1.40);
 \draw[<->, very thick, color=red] (1.260) .. controls +(260:0.5) and +(160:0.5) .. (2.160);
 \draw[<->, very thick, color=red] (2.20) .. controls +(20:0.5) and +(280:0.5) .. (3.280);
 \draw[->, very thick, color=blue] (1) -- (2);
 \node[below of=2, yshift=8mm] {(b)};
\end{scope}
 \end{tikzpicture}
 \caption{Two mDAGs whose corresponding models are the same under any
   set of perfect node interventions.}
\label{fig:3cycle2}
\end{center}
\end{figure}

\end{rmk}

\section{Discussion} \label{sec:discuss}

The class of mDAGs provides a natural framework to represent the
margins of non-parametric Bayesian network models, and the structure
of these models under interventions when interpreted causally.  We
have given a partial characterization of the Markov equivalence class
of these models under the marginal Markov property, but a full result
is still an open problem.  As mentioned in Section \ref{sec:markov},
Markov equivalence for the nested Markov model is also open.

Fitting and testing models under the marginal Markov property is
difficult because no explicit representation of the model is
generally available, though the results in Section \ref{sec:markov}
give characterizations in special cases (see Example \ref{exm:equiv}).
The work of \citet{bonet:01} suggests that a general characterization
may be infeasible because of the complexity of the inequality
constraints.  The nested model provides a useful surrogate because, at
least in the discrete case, it is known to be smooth, has an explicit
parameterization, and has the same dimension as the marginal model
\citep{evans:complete}.  Since $\M_n(\G) \supseteq \M_m(\G)$, if the
nested model is a bad fit then so is the marginal model.  The converse
is not true however, so we potentially lose power by ignoring
inequality constraints.  \citet{evans:12} gives a graphical method for
deriving some inequality constraints, so these can in principle be
tested after fitting a larger model.  The approach of
\citet{richardson:11} gives a parameterization of the marginal model
for the mDAG in Figure \ref{fig:inst2}(a), incorporating inequality
constraints; a general parameterization for such models is another
open problem.

Alternatively it is possible to use a latent variable model $\M_l(\G)$
as a second surrogate, knowing that $\M_l(\G) \subseteq \M_m(\G)$.  If
the nested and latent variable models give similar fits (by some
suitable criterion) then we effectively have a fit for the marginal
model, which lies in between the two.  Methods for fitting models
under the marginal Markov property would enable powerful search
procedures for distinguishing between different causal models with
latent variables.

\subsection*{Acknowledgements}

We thank Steffen Lauritzen for helpful discussions, and two anonymous
referees for excellent suggestions, including the idea of using a
simplicial complex to represent the bidirected structure.

\newpage

\bibliographystyle{plainnat}
\bibliography{mybib}

\appendix

\section{Technical Proofs} \label{sec:proofs}

\subsection{Proof of Theorem \ref{thm:proj}}

\begin{lem} \label{lem:proj}
Let $\G(V \dot\cup U_1 \dot\cup U_2, \mathcal{E}_\G, \mathcal{B}_\G)$ be an mDAG, and $\H(V \dot\cup U_1, \mathcal{E}_{\H}, \mathcal{B}_{\H})$ the latent projection of $\G$ over $V \dot\cup U_1$.  Then
\begin{enumerate}[(a)]
\item for $a,b \in V$, there is a directed path from $a$ to $b$ in $\G$ with non-endpoint 
vertices in $U_1 \dot\cup U_2$ if and only if there is such a path in $\H$ with non-endpoint vertices 
in $U_1$;
\item there is a hidden common cause for $B \subseteq V$ in $\G$ with
  respect to $U_1 \dot\cup U_2$ if and only if there is a hidden common cause for $B$
  in $\H$ with respect to $U_1$.  
\end{enumerate}
\end{lem}

\begin{proof}
  (a): Suppose there is a directed path from $a$ to $b$ in $\G$ with
  non-endpoint vertices in $U_1 \cup U_2$.  If any non-endpoint vertices on the
  path are also in $U_1$, then the problem reduces to showing the
  existence of two shorter paths (acyclicity means we can always
  concatenate directed paths and still obtain a path).  On the other
  hand if all non-endpoint vertices are in $U_2$ then there
  is an edge $a \rightarrow b$ in $\H$.  

  Conversely if there is a directed path in $\H$ with intermediate vertices in 
  $U_1$ then each edge $c \rightarrow d$ in that path represents a directed path from $c$ to
  $d$ in $\G$ with intermediate vertices in $U_2$.

  (b): Let $B \subseteq V$ have a hidden common cause in $\G$ with
  respect to $U_1 \cup U_2$; for each $b \in B$ there is a directed path $\pi_b$
  to $b$ with all other vertices in $U_1 \cup U_2$ as described in the definition
  of a hidden common cause.  Let $u_b$ be the first vertex on $\pi_b$
  which is not in $U_2$ (certainly $b \notin U_2$, so this is well defined).
  Then the vertices $A = \{u_b : b \in B\}$ share a hidden common
  cause with respect to $U_2$, and hence $A \in
  \mathcal{B}_\H$.

  But for each $b \in B$, there is a directed path in $\G$ from $u_b$
  to $b$ with non-endpoints in $U_1 \cup U_2$, and hence by (a) there is a
  directed path in $\H$ from $u_b$ to $b$ with non-endpoints in
  $U_1$; hence the vertices in $B$ share a hidden common cause
  with respect to $U_1$ in $\H$.

  Conversely, suppose the elements of $B$ share a hidden common cause
  $A \in \mathcal{B}_\mathcal{H}$ with respect to $U_1$ in $\H$.  By the definition of latent
  projection, the vertices in $A$ must share a hidden common cause
  $C$ with respect to $U_2$ in $\G$.  It follows by
  concatenating the paths from $C$ to $A$, and
  from $A$ to $B$, that the vertices in $B$ share the hidden common
  cause $C$ with respect to $U_1 \cup U_2$ in $\G$.
\end{proof}

\begin{proof}[Proof of Theorem \ref{thm:proj}]
  It is sufficient to prove the first equality: let $\H =
  \mathfrak{p}(\G, V \cup U_1)$.  Let $a,b \in V$; by Lemma
  \ref{lem:proj}, there is a directed path from $a$ to $b$ in $\G$
  with all non-endpoint vertices in $U_1 \cup U_2$ if and only if
  there is such a path in $\H$ with all non-endpoint vertices
  in $U_1$.  Hence the directed edges in $\mathfrak{p}(\G, V)$ and
  $\mathfrak{p}(\H, V)$ are the same.

  Also by Lemma \ref{lem:proj}, for any set $B \subseteq V$, there is
  a hidden common cause in $\G$ for $B$ with respect to $U_1 \cup
  U_2$, if and only if there is one in $\H$ for $B$ with respect to
  $U_1$.  Hence the bidirected faces in $\mathfrak{p}(\G, V)$ and
  $\mathfrak{p}(\H, V)$ are also the same.
\end{proof}

\subsection{Measure Theoretic Results}

Let $X$ be a square integrable random variable, and $\mathcal{F}$ a 
$\sigma$-algebra.  Say that $X$ is $(\epsilon, \mathcal{F})$-measurable 
if $\E(X - \E[X \,|\, \mathcal{F}])^2 \leq \epsilon$

Let $\mathcal{F}^{-i} \equiv \mathcal{F}_1 \vee \cdots \vee \mathcal{F}_{i-1} \vee 
\mathcal{F}_{i+1}\vee \cdots\vee\mathcal{F}_{k}$.

\begin{lem} \label{lem:meas}
Let $X_i$ be $(\epsilon, \mathcal{F}^{-i})$-measurable for $i=1,\ldots,k$, where 
$\mathcal{F}_j$ are independent $\sigma$-algebrae.  

Then $\E(X_i - X_j)^2 \leq \epsilon$ for all $i,j$ implies that $X_i$ is
$(2\epsilon, \mathcal{F}^{-i,j})$-measurable for $i\neq j$.  In addition,
$\Var X_i \leq k \epsilon$.
\end{lem}

\begin{proof}
Since $X_i, \mathcal{F}^{-i} \indep \mathcal{F}_i$,
\begin{align*}
\E(X_i - \E[X_i \,|\, \mathcal{F}^{-i,j}])^2 &= \E(X_i - \E[X_i \,|\, \mathcal{F}^{-j}])^2\\
  &\leq \E(X_i - \E[ X_j \,|\, \mathcal{F}^{-j}])^2\\
  &\leq \E(X_i - X_j)^2 + \E(X_j - \E[ X_j \,|\, \mathcal{F}^{-j}])^2\\
  &\leq 2\epsilon,
\end{align*}
so $X_i$ is $(2\epsilon, \mathcal{F}^{-i,j})$-measurable.  Repeating this proof 
shows that $X_i$ is $(k\epsilon, \emptyset)$-measurable, which is to say that 
its variance is at most $k\epsilon$.
\end{proof}





\begin{lem} \label{lem:split}
  Let $X$ be a $\sigma(Y,Z)$-measurable random variable, and $(X,Y,Z)$
  have joint distribution $P$.  Then there exist random variables $U,
  W$ such that: 
\begin{enumerate}[(i)]
\item $U \indep W$;
\item $X$ is $\sigma(Y,U)$-measurable;
\item $Z$ is $\sigma(W,X,Y)$-measurable;
\item $(X, Y, Z)$ has the appropriate joint distribution $P$.  
\end{enumerate}
\end{lem}

\begin{proof}
  Using the fact that our probability space is Lebesgue-Rokhlin, there
  exists a measurable function $g$ such that if $U$ is a uniform
  random variable independent of $Y$ then $(X,Y) \equiv (g(Y, U), Y)$
  has the correct marginal distribution \citep[][Theorem
  2.2]{chentsov:82}.  Similarly, let $W$ be a uniform random variable
  independent of $U,Y$ (and therefore $X$), and let $h$ be a
  measurable function such that $(X,Y,Z) \equiv (X,Y,h(X,Y,W))$ has
  the same distribution as $(X,Y,Z)$.

  By construction, (i)-(iv) are satisfied.  
\end{proof}





\begin{lem} \label{lem:splitnode} Let $\G$ be an mDAG containing a
  bidirected facet $B = C \dot\cup D$ such that: for any $c \in C$, any
  bidirected edge containing $c$ is a subset of $B$; and $\pa_\G(d)
  \supseteq \pa_\G(C)$ for each $d \in D$.

  Take $P \in \M_m(\G)$.  Then there exists $Q \in \M(\bar{\G})$ such
  that under $Q$ we have $Y_B = (Y_C, Y_D)$, where:

\begin{enumerate}[(i)]
\item $Y_C \indep Y_D$;
\item each $X_c$ is $\sigma(X_{\pa_\G(c)}, Y_C)$-measurable
\item each $X_d$ is $\sigma(X_C, X_{\pa_\G(C)}, X_{\pa_\G(d)}, Y_{\mathcal{B}(d) \setminus B}, Y_D)$-measurable;
\item the $V$-margin of $Q$ is $P$.
\end{enumerate}
\end{lem}

\begin{proof}
  This is just an application of Lemma \ref{lem:split} with $X = X_C$,
  $Y = X_{\pa_\G(C)}$, $Z = X_D$, and some extra variables
  $X_{\pa_\G(d)}, Y_{\mathcal{B}(d) \setminus B}$ on which $Z$ can
  depend (but this extension is trivial).
\end{proof}

In other words, the result says that we can decompose $Y_B$ into two
independent pieces, one of which determines the value of $X_C$ (once
its parents are known) and contains no further information, in the
sense that it is irrelevant once $X_C$ and $X_{\pa(C)}$ are known.

\end{document}